\documentclass[12pt]{amsart}
\usepackage{amsmath}
\usepackage{amsthm}
\usepackage{amsfonts}
\usepackage{amssymb}
\usepackage[all,cmtip]{xy}           
\usepackage{bm}
\usepackage{bbm}
\usepackage{bbding}
\usepackage{txfonts}
\usepackage{amscd}
\usepackage{stmaryrd}
\usepackage{xspace}
\usepackage[shortlabels]{enumitem}
\usepackage{ifpdf}
\usepackage{graphicx}
\usepackage{stmaryrd}
\usepackage[all]{xy}
\ifpdf
  \usepackage[colorlinks,final,backref=page,hyperindex]{hyperref}
\else
  \usepackage[colorlinks,final,backref=page,hyperindex,hypertex]{hyperref}
\fi
\usepackage[active]{srcltx}
\usepackage{tikz-cd}
\usepackage{tikz}

\usetikzlibrary{positioning,fit}
\makeatletter

\newtheorem{thm}{Theorem}[section]
\newtheorem{lem}[thm]{Lemma}
\newtheorem{cor}[thm]{Corollary}
\newtheorem{pro}[thm]{Proposition}
\theoremstyle{definition}

\newtheorem{rmk}[thm]{Remark}
\newtheorem{defi}[thm]{Definition}

\newcommand{\nc}{\newcommand}
\newcommand{\delete}[1]{}
\newcommand{\ol}{\overline}

\nc{\mlabel}[1]{\label{#1}}  
\nc{\mcite}[1]{\cite{#1}}  
\nc{\mref}[1]{\ref{#1}}  
\nc{\meqref}[1]{\eqref{#1}}  
\nc{\mbibitem}[1]{\bibitem{#1}} 

\delete{
\nc{\mlabel}[1]{\label{#1}{\hfill \hspace{1cm}{\bf{{\ }\hfill(#1)}}}}
\nc{\mcite}[1]{\cite{#1}{{\bf{{\ }(#1)}}}}  
\nc{\mref}[1]{\ref{#1}{{\bf{{\ }(#1)}}}}  
\nc{\meqref}[1]{\eqref{#1}{{\bf{{\ }(#1)}}}}  
\nc{\mbibitem}[1]{\bibitem[\bf #1]{#1}} 
}

\DeclareMathOperator{\im}{Im}

\newcommand {\emptycomment}[1]{}

\delete{

}

\nc{\oprn}{\theta}

\delete{
\setlength{\baselineskip}{1.8\baselineskip}

\newcommand{\emptycomment}[1]{}

}

\nc{\calo}{\mathcal{O}}
\nc{\oop}{$\mathcal{O}$-operator\xspace}
\nc{\oops}{$\mathcal{O}$-operators\xspace}
\nc{\mrho}{{\bm{\varrho}}}
\nc{\bfk}{\mathbf{K}}
\nc{\invlim}{\displaystyle{\lim_{\longleftarrow}}\,}
\nc{\ot}{\otimes}

\nc{\eval}[1]{\Big|_{#1}}

\newcommand{\be }{\begin{equation}}
\newcommand{\ee }{\end{equation}}

\nc{\RR}{\mathbb{R}}

\nc{\hC}{\mathcal{C}}

\newcommand{\hB}{\mathcal{B}}

\newcommand{\Ad}{\operatorname{Ad}}

\newcommand{\frkg}{\mathfrak g}
\newcommand{\frkh}{\mathfrak h}

\newcommand{\br}[1]{   [ \cdot,    \cdot  ]   }
\newcommand{\id}{\mathsf{id}}

\newcommand{\wtd}{\widetilde}
\nc{\CV}{\mathbf{C}}

\newcommand{\bhd}{\blacktriangleright}

\NewDocumentEnvironment{Thm}{O{thm} D(){} m}
  {\addtocounter{#1}{-1}%
   \expandafter\renewcommand\csname the#1\endcsname{\ref{#3}}%
   \begin{#1}[#2]}
  {\end{#1}}

\begin{document}

\title[Matched pairs of Lie algebras and Rota-Baxter Lie algebras]{Matched pairs of Lie algebras and Rota-Baxter Lie algebras}

\author{Shukun Wang}
\address{School of Mathematics and Big Data, Anhui University of Science and Technology, Huainan 232001, China; Anhui Province Engineering Laboratory for Big Data Analysis and Early Warning Technology of Coal Mine Safety, Huainan 232001, China}
\email{2024093@aust.edu.cn}


\begin{abstract}
In this paper, we investigate the relationship between matched pairs of Lie algebras and Rota-Baxter Lie algebras. First, we show that every Rota-Baxter Lie algebra $(\frkg,B)$ of weight $-1$ gives rise to a matched pair of Lie algebras $(\frkg_+,\frkg_-,\rhd,\bhd)$, and we prove that the bicrossed product of $\frkg_+$ and $\frkg_-$ decomposes as $\frkg_+\bowtie\frkg_-=\frkg_1\oplus\frkg_2$. Moreover, we establish a Rota-Baxter Lie algebra structure on $\frkg_1$ which is isomorphic to $(\frkg,B)$ as a Rota-Baxter Lie algebra, and we endow $\frkg_2$ with a Rota-Baxter Lie algebra structure that is Rota-Baxter homomorphic to the descendent Rota-Baxter Lie algebra $(\frkg_B,B)$. Then we study the connection between quadratic Rota-Baxter Lie algebras and Manin triples. Finally, we show that every Rota-Baxter group induces a matched pair of groups and investigate the internal structure of the induced matched pair of groups.

\end{abstract}


\keywords{matched pairs of Lie algebras, Rota-Baxter Lie algebras, quadratic Rota-Baxter Lie algebras, matched pairs of groups, Rota-Baxter groups. \\
\quad  2020 \emph{Mathematics Subject Classification.} 17B38, 17B62.}

\maketitle

\tableofcontents

\section{Introduction}

Rota-Baxter operators on associative algebras originated in the work of G.-C. Rota \cite{RT} and Baxter \cite{BX} in probability theory and combinatorics. They also play an important role in Connes-Kreimer’s algebraic approach to the renormalisation of quantum field theory \cite{CO}. 

Independently, Rota-Baxter operators on Lie algebras have also been investigated by mathematical physicists in connection with the operator form of the (modified) classical Yang-Baxter equation. The systematic study in this direction was initiated by Semenov-Tian-Shansky \cite{STS}. Let $\frkg$ be a Lie algebra.
A linear operator $B:\frkg\to \frkg$ is called a Rota-Baxter operator of weight $\lambda$ if it satisfies
$$[B(x),B(y)]=B([B(x),y]+[x,B(y)]+\lambda [x,y]), \ \forall x,y\in \frkg. $$ 
The pair $(\frkg,B)$ is called a Rota-Baxter Lie algebra of weight $\lambda.$
In particular, Rota-Baxter operators of weight $0$ can be regarded as the operator form of the classical Yang-Baxter equation, while those of weight $1$ are in one-to-one correspondence with solutions of the modified Yang-Baxter equation and give rise to the Infinitesimal Factorisation Theorem for Lie algebras. Moreover, the
Global Factorisation Theorem was obtained by integrating the Infinitesimal Factorisation Theorem for a Lie algebra locally, it has many applications in integrable systems; see \cite{FR,GR,GR1,S2}.

To derive the global factorisations of Semenov-Tian-Shansky on the Lie group level directly, in 2021, Rota-Baxter operators on (Lie) groups were defined in \cite{LG}. Differentiating a Rota-Baxter operator on Lie groups, one can get a Rota-Baxter operator of weight $1$ on a Lie algebra. Rota-Baxter operators on (Lie) groups have been extensively studied in recent years; see \cite{VG,FC,DA}.  

Matched pairs of Lie algebras and matched pairs of groups originated in work related to Poisson geometry and dressing transformations, particularly in the context of Drinfeld doubles and Poisson-Lie groups; see \cite{BD,D,KO,LU,MJ}. They provide algebraic frameworks for decomposing Lie bialgebras and for decomposing Poisson-Lie groups.

To understand the operator forms
of non-skew-symmetric solutions of the classical Yang-Baxter equation better, quadratic Rota-Baxter Lie algebras were introduced in \cite{Hl}. A quadratic Rota-Baxter Lie algebra of weight $\lambda$ is a triple $(\frkg,B,S)$ such that $B:\frkg\to\frkg$ is a Rota-Baxter operator of weight $\lambda$ and $(\frkg,S)$ is a quadratic Lie algebra satisfying 
 \begin{equation}\label{EQRB}
 S(B(x),y)+S(x,B(y))+\lambda S(x,y)=0, \ \forall x,y\in \frkg.
 \end{equation}
There is a one-to-one correspondence between quadratic Rota-Baxter Lie algebras of nonzero weight and factorizable Lie bialgebras.

Let $(\frkg,B)$ be a Rota-Baxter Lie algebra of weight $-1.$ Then there is a Lie algebra structure $\frkg_B$ on $\frkg$ whose Lie bracket $[\cdot,\cdot]_{B}:\frkg\times \frkg\to \frkg$ is given by 
$$[x,y]_{B}=[B(x),y]+[x,B(y)]+\lambda [x,y],\ \forall x,y\in \frkg.$$
The pair $(\frkg_B,B)$ is also a Rota-Baxter Lie algebra of weight $-1.$
The matched pairs of Lie algebras on $\frkg\oplus \frkg_B$ were investigated in \cite{Hl}. Let $\wtd{B}= \id-B$, denote $\frkg_+=\im B$ and $\frkg_-=\im\wtd{B}.$ The Infinitesimal Factorisation Theorem for Lie algebras states that there is a decomposition of $\frkg$ to $\frkg_+\oplus \frkg_-.$ However, the matched pairs of Lie algebras on $\frkg_+\oplus \frkg_-$ are still mysterious. Our first goal is to understand Semenov-Tian-Shansky’s Infinitesimal Factorisation Theorem for Lie algebras from the perspective of matched pairs of Lie algebras. We show that there is a matched pair of Lie algebras on $\frkg_+\oplus \frkg_-,$ which we call the matched pairs of Lie algebras on $(\frkg,B).$ Furthermore, we prove that this matched pair can be decomposed into the direct sum $\frkg_1\oplus \frkg_2$ of Lie algebras. We show that there exists a Rota-Baxter Lie algebra on $\frkg_1$ that is Rota-Baxter isomorphic to $(\frkg,B)$. Moreover, there is a Rota-Baxter Lie algebra $\frkg_2$ that is Rota-Baxter isomorphic to the Rota-Baxter quotient Lie algebra of $(\frkg_B,B).$

For a quadratic Rota-Baxter Lie algebra of weight $-1$  $(\frkg,B,S)$, the identity \eqref{EQRB} can be written as 
$$S(B(x),y)+S(x,\wtd{B}(y))=0,\ \forall x,y\in \frkg. $$ This means a quadratic Rota-Baxter Lie algebra of weight $-1$ gives rise to a nondegenerate symmetric bilinear form $S'$ on $\frkg_+\oplus\frkg_-.$ Based on the matched pair of Lie algebras on $(\frkg,B),$ our second goal is to construct the Manin triple on $\frkg_+\oplus \frkg_-.$ We show that there is a Manin triple structure on $\frkg_+\oplus \frkg_-,$ called the Manin triple on $(\frkg,B,S).$ Then we prove that this Manin triple can be decomposed into the direct sum of two quadratic Lie algebras. Furthermore, we study the induced Rota-Baxter structures on these quadratic Lie algebras. 

Our third goal is to interpret Semenov-Tian-Shansky’s Global Factorisation Theorem for Lie groups within the framework of matched pairs of groups. We prove that every Rota-Baxter group of weight $-1$ $(G,\hB)$  induces a matched pair of groups, which we call the matched pair of groups on $(G,\hB).$ Then we consider the relationship between the matched pair of groups on a Rota-Baxter group of weight $-1$ and the matched pair of Lie algebras on a Rota-Baxter Lie algebra of weight $-1.$ Finally, we investigate the internal structure of the matched pair of groups on $(G,\hB)$.

This paper is organised as follows. In Section \ref{2}, we study the matched pair of Lie algebras on a Rota-Baxter Lie algebra of weight $-1$. We prove that this matched pair can be decomposed into the direct sum of two Lie algebras, investigating the Rota-Baxter Lie algebra structures on these Lie algebras. In Section \ref{3}, we study the Manin triple on a quadratic Rota-Baxter Lie algebra of weight $-1$ and give a decomposition theorem on this Manin triple. In Section \ref{4}, we introduce the notion of the matched pair of groups on a Rota-Baxter group of weight $-1$ and investigate its inner structure.

 \section{Matched pairs of Lie algebras, Rota–Baxter Lie algebras and Lie algebra projections}\label{2}

This section is divided into two parts. In the first part, we recall the basic notions of matched pairs of Lie algebras and Rota-Baxter Lie algebras, and show that every Rota-Baxter Lie algebra of weight $-1$ induces a matched pair of Lie algebras. In the second part, we propose the definition of Lie algebra projections on matched pairs of Lie algebras and prove two results: (\romannumeral1) every Lie algebra projection $C$ on a matched pair of Lie algebras $(\frkg_+,\frkg_-,\rhd,\bhd)$ gives rise to a Rota-Baxter Lie algebra structure on $\ker C$; (\romannumeral2) every Rota-Baxter Lie algebra of weight 
$-1$ induces a Lie algebra projection on the corresponding matched pair. 

In this section, let $(\frkg,B)$ be a Rota-Baxter Lie algebra of weight $\lambda$($\lambda\neq 0$). Denote $\im B$ by $\frkg_+$ and $\im \wtd{B}$ by $\frkg_-.$ And denote $\ker \wtd{B}$ by $\frkh_+$ and $\ker B$ by $\frkh_-.$ It follows directly from the definition that $\frkg_+,\frkg_-$ are Lie algebras and $\frkh_+\subseteq \frkg_+,$ $\frkh_-\subseteq \frkg_-$ are ideals.

\subsection{The Mached pair of Lie algebras on a Rota-Baxter Lie algebra}
First, let us recall the definition of Rota-Baxter Lie algebras. Let $\frkg$ be a Lie algebra and $B:\frkg\to \frkg$ be an operator on $\frkg.$ Recall that $B$ is called a Rota-Baxter operator of weight $\lambda$ on $\frkg$ if it satisfies 
\begin{equation}\label{RBLRS}[B(x),B(y)]=B([B(x),y]+[x,B(y)]+\lambda [x,y]),\ \forall x,y\in \frkg.
\end{equation}
Furthermore, $(\frkg,B)$ is called a Rota-Baxter Lie algebra of weight $\lambda.$ It is well-known that the operator $\wtd{B}=-\lambda-B$ is also a Rota-Baxter operator on $\frkg.$ 

Then let us recall the definition of Rota-Baxter Lie algebra homomorphisms. Let $(\frkg,B)$ and $(\frkg', B')$ be Rota-Baxter Lie algebras of weight $\lambda$. Then a map $f:(\frkg,B)\to (\frkg', B')$ is called a Rota-Baxter Lie algebra homomorphism if $f$ is a Lie algebra homomorphism from $\frkg$ to $\frkg'$ satisfying
\begin{equation}\label{RBLH}
f\circ B=B' \circ f.
\end{equation}

Let $(\frkg,B)$ be a Rota-Baxter Lie algebra of weight $\lambda.$ Recall that the operator $[\cdot,\cdot]_{B}:\frkg\times\frkg\to \frkg$ by 
\begin{equation*}
[x,y]_{B}=[B(x),y]+[x,B(y)]+\lambda [x,y],\ \forall x,y\in \frkg.
\end{equation*}
is a Lie bracket, and the Lie algebra $(\frkg,[\cdot,\cdot]_B)$ is called the descendent Lie algebra of $(\frkg,B)$.
 
The next Lemma will be used in the proof of \ref{RLMP}.
\begin{lem}\label{RBRL}
Let $(\frkg,B)$ be a Rota-Baxter Lie algebra of weight $\lambda$ $(\lambda\neq 0)$. Then $[\cdot,\cdot]_{B}=-[\cdot,\cdot]_{\wtd{B}}.$
\end{lem}
\begin{proof}
For any $x,y\in \frkg,$ we have 
\begin{equation}\label{BTB}
[x,y]_{B}=[B(x),y]+[x,B(y)]+\lambda[x,y]=[B(x),y]-[x,\wtd{B}(y)].
\end{equation}
Then by the antisymmetry of $[\cdot,\cdot]_B,$ we have 
\begin{equation*}
[x,y]_B=-[y,x]_{B}=[x,B(y)]-[\wtd{B}(x),y]=-[x,y]_{\wtd{B}}.
\end{equation*}
Hence $[\cdot,\cdot]_{B}=-[\cdot,\cdot]_{\wtd{B}}.$
\end{proof}

Let $(\frkg,B)$ be a Rota-Baxter Lie algebra of nonzero weight. As $\frkh_+\subseteq \frkg_+$ and $\frkh_-\subseteq \frkg_-$ are ideals, we have the following lemma.
\begin{lem}\label{RBSY}
Let $(\frkg, B)$ be a Rota-Baxter Lie algebra of weight $\lambda$ $(\lambda\neq 0)$. Then $\frkh_-$ is invariant under the adjoint action of $\frkg_-$, that is, 
$$[\wtd{B}(x),y]\in \frkh_-,\ \forall x\in \frkg,\ \forall y\in \frkh_-.$$ Respectively, $\frkh_+$ is invariant under the adjoint action of $\frkg_+,$ that is,  
$$[B(x),y]\in \frkh_+,\ \forall x\in \frkg,\ \forall y\in \frkh_+.$$
\end{lem}

Next, let us recall the notion of matched pairs of Lie algebras.

\begin{defi}
A matched pair of Lie algebras is a quadruple $(\frkg_+,\frkg_-,\rhd,\bhd)$, where $\frkg_+$ and $\frkg_-$ are Lie algebras, $\rhd:\frkg_+\times \frkg_-\to \frkg_-$ is a representation of $\frkg_+$ on $\frkg_-$, $\bhd:\frkg_-\times \frkg_+\to \frkg_+$ is a representation of $\frkg_-$ on $\frkg_+$ such that the following conditions hold:
\begin{subequations}
\begin{equation}\label{MP1}
x\rhd[u,v]_{\frkg_-}=[x\rhd u,v]_{\frkg_-}+[u,x\rhd v]_{\frkg_-}+(v\bhd x)\rhd u-(u\bhd x)\rhd v,
\end{equation}

\begin{equation}\label{MP2}
u\bhd [x,y]_{\frkg_+}=[u\rhd x,y]_{\frkg_+}+[x,u\rhd y]_{\frkg_+}
+ (y\rhd u)\bhd x-(x\rhd u)\bhd y, 
\end{equation}
\end{subequations}
for any $x,y \in \frkg_+$ and $u,v\in \frkg_-.$

For a matched pair of Lie algebras $(\frkg_+,\frkg_-,\rhd,\bhd)$, there is a Lie algebra structure on $\frkg_+\oplus \frkg_-$ with the Lie bracket $[\cdot,\cdot]_{\bowtie}:(\frkg_+\oplus \frkg_-)\times (\frkg_+\oplus \frkg_-)\to \frkg_+\oplus \frkg_-$ given by 
$$[(x,u),(y,v)]_{\bowtie}=([x,y]_{\frkg_+}+u\rhd y-v\rhd x,[u,v]_{\frkg_-}+x\bhd v-y\bhd u),\ \forall (x,u), (y,v)\in \frkg_+\oplus \frkg_-.$$
This Lie algebra is called the bicrossed product of Lie algebras $\frkg_+$ and $\frkg_-.$
\end{defi}

\begin{rmk}\label{MM}
Let $\frkg_+$ and $\frkg_-$ be Lie algebras. If there is a Lie bracket $[\cdot,\cdot]:\frkg_+\oplus \frkg_-\to \frkg_+\oplus \frkg_-$ where $\frkg_+\oplus \frkg_-$ denotes the direct sum of $\frkg_+$ and $\frkg_-$ as vector spaces. Define the operator $\rhd:\frkg_+\times \frkg_-\to \frkg_-$ by 
$$x\rhd u=p_- [x,u],\ \forall x\in \frkg_+, \ \forall u\in \frkg_-
$$
and the operator $\bhd:\frkg_-\times \frkg_+\to \frkg_+$ by 
$$ u\bhd x =p_+ [x,u],\ \forall x\in \frkg_+,\ \forall u\in \frkg_-,$$ where $p_+:\frkg_+\oplus \frkg_-\to \frkg_+$ and $p_-:\frkg_+\oplus \frkg_-\to \frkg_-$ are projections. Then $(\frkg_+,\frkg_-,\rhd,\bhd)$ is a matched pair of Lie algebras.

\end{rmk}

Then let us recall the definition of homomorphisms of matched pairs of Lie algebras.

\begin{defi}
Let $(\frkg_+,\frkg_-,\rhd,\bhd)$ and $(\frkg'_+,\frkg'_-,\rhd', \bhd')$ be two matched pairs of Lie algebras. A homomorphism of matched pairs of Lie algebras is a double $(f_+,f_-)$ such that:
\begin{itemize}
\item[(a)] $f_+:\frkg_+\to \frkg_+$ and $f_-:\frkg_-\to \frkg_-$ are Lie algebra homomorphisms;

\item[(b)] $f_+$ and $f_-$ satisfy
\begin{equation}\label{MPLH}
\begin{aligned}
&f_+(u\bhd x)=f_-(u)\bhd' f_+(x),\ \forall x\in \frkg_+,\ \forall u\in \frkg_-;\\
&f_-(x\rhd u)=f_+(x)\rhd' f_-(u), \forall x\in \frkg_+,\ \forall u\in \frkg_-.\\
\end{aligned}
\end{equation}
\end{itemize}
\end{defi}

Let $(\frkg, B)$ be a Rota-Baxter Lie algebra of weight $-1$. Denote $\im B$ by $\frkg_+$ and $\im\wtd{B}$ by $\frkg_-.$ It is straightforward to see that $\frkg_+$ and $\frkg_-$ are both Lie subalgebras of $\frkg.$ 
Define the map
$\rhd: \frkg_+\times\frkg_-\to \frkg_-$ by 
\begin{equation*}
B(x)\rhd\wtd{B}(y)=\wtd{B}([B(x),y]),\ \forall x,y\in \frkg.
\end{equation*}
 And define the map $\bhd:\frkg_- \times \frkg_+\to \frkg_+$ by
\begin{equation*}\wtd{B}(x)\bhd B(y)={B([\wtd{B}(x),y])},\ \forall x,y\in \frkg.\end{equation*}
By Lemma \ref{RBSY}, one can check that $\rhd$ and $\bhd$ are both well defined. It follows directly from the definition that $\rhd$ is a representation of $\frkg_+$ on $\frkg_-$ and $\bhd$ is a representation of $\frkg_-$ on $\frkg_+$.

Now, we prove that every Rota-Baxter Lie algebra of weight $-1$ gives rise to a matched pair of Lie algebras.

\begin{thm}\label{RLMP}
Let $(\frkg,B)$ be a Rota-Baxter Lie algebra of weight $\lambda$ $(\lambda\neq 0)$. With the above notations, then $(\frkg_+,\frkg_-,\rhd,\bhd)$ is a matched pair of Lie algebras.
\end{thm}

\begin{proof}
For any $x,y,z\in \frkg$, we have 
\begin{equation*}
\begin{aligned}
\relax &[B(x)\rhd \wtd{B}(y),\wtd{B}(z)]+[\wtd{B}(y),B(x)\rhd \wtd{B}(z)]+(\wtd{B}(z)\bhd B(x))\rhd \wtd{B}(y)-(\wtd{B}(y)\bhd B(x))\rhd \wtd{B}(z)\\
=&[\wtd{B}([B(x),y]),\wtd{B}(z)]+[\wtd{B}(y),\wtd{B}([B(x),z])]+\wtd{B}([B([\wtd{B}(z),x]),y])-\wtd{B}([
B([\wtd{B}(y),x]),z])\\
=&-\wtd{B}([B([B(x),y]),z])+\wtd{B}([[B(x),y],\wtd{B}(z)])+\wtd{B}([\wtd{B}(y),[B(x),z]])-\wtd{B}([y,B([B(x),z])])\\&+\wtd{B}([B([\wtd{B}(z),x]),y])-\wtd{B}([B([\wtd{B}(y),x]),z])\  (\textit{by \eqref{RBLRS} and \eqref{BTB}})\\
=&-\wtd{B}([[B(x),B(y)],z])+\wtd{B}([[B(x),B(z)], y])+\wtd{B}([[B(x),y],\wtd{B}(z)])+\wtd{B}([\wtd{B}(y),[B(x),z]])\\  &(\textit{by \eqref{RBLRS}, \eqref{BTB} and Lemma \ref{RBRL}}).
\end{aligned}
\end{equation*}

Then we have 
\begin{equation*}
\begin{aligned}
\relax &[B(x)\rhd \wtd{B}(y),\wtd{B}(z)]+[\wtd{B}(y),B(x)\rhd \wtd{B}(z)]+(\wtd{B}(z)\bhd B(x))\rhd \wtd{B}(y)-(\wtd{B}(y)\bhd B(x))\rhd \wtd{B}(z)\\
=&\wtd{B}([[B(y),z],B(x)])+\wtd{B}([[z,B(x)],B(y)])+\wtd{B}([[B(x),B(z)], y])+\wtd{B}([[B(x),y],\wtd{B}(z)])\\&+\wtd{B}([\wtd{B}(y),[B(x),z]])\ (\textit{by the Jacobi identity})\\
=&-B(x)\rhd \wtd{B}([B(y),z])+\lambda\wtd{B}([[B(x),z],y]) +\wtd{B}([[B(x),B(z)], y])+\wtd{B}([[B(x),y],\wtd{B}(z)])\\
=&-B(x)\rhd \wtd{B}([B(y),z])+\wtd{B}([[\wtd{B}(z),B(x)],y])+\wtd{B}([[B(x),y],\wtd{B}(z)])
.
\end{aligned}
\end{equation*}
And then we have \begin{equation*}
\begin{aligned}
\relax &\wtd{B}([B(x)\rhd \wtd{B}(y),\wtd{B}(z)])+\wtd{B}([\wtd{B}(y),B(x)\rhd \wtd{B}(z)])+(\wtd{B}(z)\bhd x)\rhd \wtd{B}(y)-(\wtd{B}(y)\bhd x)\rhd \wtd{B}(z)\\
=&-B(x)\rhd \wtd{B}([B(y),z])+\wtd{B}([B(x),[y,\wtd{B}(z)]])\ (\textit{by the Jacobi identity})\\
=&-B(x)\rhd \wtd{B}([B(y),z])+B(x)\rhd \wtd{B}([y,\wtd{B}(z)])\\
=&B(x)\rhd [\wtd{B}(y),\wtd{B}(z)]_{\frkg_-}.
\end{aligned}
\end{equation*}
This proves that \eqref{MP1} holds, and it is similar to show that \eqref{MP2} holds. Therefore, $(\frkg_+,\frkg_-,\rhd,\bhd)$
is a matched pair of Lie algebras.
\end{proof}

The matched pair of Lie algebras $(\frkg_+,\frkg_-,\rhd,\bhd)$ given in Theorem \ref{RLMP} is called \textbf{the matched pair of Lie algebras on $(\frkg,B).$}

Let $(\frkg, B)$ be a Rota-Baxter Lie algebra of weight $\lambda$ and $(\frkg_+,\frkg_-,\rhd,\bhd)$ be the matched pair of Lie algebras given in the above proposition. Then $[\cdot ,\cdot]_{\bowtie}$ is given by 
\begin{equation}\label{ID1}
\begin{aligned}
&[\left(B(x_1),\wtd{B}(x_2)\right),\left(B(y_1),\wtd{B}(y_2)\right)]_{\bowtie}\\&=\left([B(x_1),B(y_1)]+B([\wtd{B}(x_2),y_1])-B([\wtd{B}(y_2),x_1]),[\wtd{B}(x_2),\wtd{B}(y_2)]+\wtd{B}([B(x_1),y_2])-\wtd{B}([B(y_1),x_2])\right),
\end{aligned}
\end{equation}
for any $x_1,x_2,y_1,y_2\in \frkg.$

Especially, if $x_1=x_2=x$ and $y_1=y_2=y$, then we have
\begin{equation}\label{MPR}
[\left(B(x),\wtd{B}(x)\right),\left(B(y),\wtd{B}(y)     \right)]_{\bowtie}=-\lambda\left( B([x,y]),\wtd{B}([x,y])     \right).
 \end{equation}

Next, we prove that every homomorphism between Rota-Baxter Lie algebras of weight $-1$ induces a homomorphism of matched pairs of Lie algebras.
\begin{pro}
Let $(\frkg,B)$ and $(\frkg', B')$ be Rota-Baxter Lie algebras of weight $-1$. Let $(\frkg_+,\frkg_-,\rhd,\bhd)$ and $(\frkg'_+,\frkg'_-,\rhd',\bhd')$ be the matched pairs of Lie algebras on $(\frkg,B)$ and $(\frkg', B')$ respectively. Let $f:(\frkg,B)\to (\frkg',B')$ be a Rota-Baxter Lie algebra homomorphism, $f_+:\frkg_+\to \frkg'_+$ be the restriction of $f$ to $\frkg_+$ and $f_-:\frkg_-\to \frkg'_-$ be the restriction of $f$ to $\frkg_-.$ Then $(f_+,f_-)$ is a homomorphism of matched pairs of Lie algebras from $(\frkg_+,\frkg_-,\rhd,\bhd)$ to $(\frkg'_+,\frkg'_-,\rhd',\bhd').$
\end{pro}
\begin{proof}
First, it is straightforward to see that $f_+$ and $f_-$ are homomorphisms of Lie algebras. Then we prove that \eqref{MPLH} holds. For any $x,y\in \frkg,$ we have 
$$
\begin{aligned}
f_+(B(x))\rhd' f_-(\wtd{B}(y))&=f(B(x))\rhd' f(\wtd{B}(y))\\&=B'(f(x))\rhd' \wtd{B}'(f(y))\\&=\wtd{B}'([B'(f(x)),f(y)])\\
&=\wtd{B}'([f(B(x)),f(y)]).\ (\textit{by \eqref{RBLH}})\\
\end{aligned}
$$
Then we have
$$
\begin{aligned}
f_+(B(x))\rhd' f_-(\wtd{B}(y))&=\wtd{B}'\circ f([(B(x)),y])\\
&=f\circ \wtd{B}([B(x),y])\\&=f_-(B(x)\rhd \wtd{B}(y)).\\
\end{aligned}
$$
It is similar to prove that $$f_-(x\rhd u)=f_+(x)\rhd' f_-(u).$$ This proves that $(f_+, f_-)$ is a homomorphism of matched pairs of Lie algebras. 
\end{proof}

Then we study matched pairs of Lie algebras on Rota-Baxter Lie algebras of weight $-1$. The following lemma will be used in the proof of Theorem \ref{FL}.

\begin{lem}\label{BBW}
Let $(\frkg,B)$ be a Rota-Baxter Lie algebra of weight $\lambda$ and $(\frkg_+,\frkg_-,\rhd,\bhd)$ be the matched pair of Lie algebras on $(\frkg,B).$ Then 
$$[(-B\circ \wtd{B}(x),\wtd{B}\circ B(x)),(-B\circ \wtd{B}(y),\wtd{B}\circ B(y))]_{\bowtie}=\lambda\left(B([\wtd{B}(x),\wtd{B}(y)]),\wtd{B}([B(x),B(y)])\right)$$
for any $x,y\in \frkg.$
\end{lem}
\begin{proof}
For any $x,y\in \frkg,$ by \eqref{RBLRS}, we have 
\begin{equation*}
\begin{aligned}
 &[(-B\circ \wtd{B}(x),\wtd{B}\circ B(x)),(-B\circ \wtd{B}(y),\wtd{B}\circ B(y))]_{\bowtie}\\
=&([B\circ\wtd{B}(x),B\circ\wtd{B}(y)]
,[\wtd{B}\circ B(y),\wtd{B}\circ B(x)])+\left((\wtd{B}\circ B(x))\bhd (B\circ \wtd{B}(y)) ,(B\circ \wtd{B}(x))\rhd (\wtd{B}\circ B(y))\right)\\&-\left(( \wtd{B}\circ B(y))\bhd (B\circ \wtd{B}(x)),(B\circ \wtd{B}(y))\rhd (\wtd{B}\circ B(x))  \right)\\
 =&\left(B([B\circ\wtd{B}(x),\wtd{B}(y)]),\wtd{B}([\wtd{B}\circ B(x),B(y)])\right)+\left(B([\wtd{B}(x),B\circ \wtd{B}(y)]),\wtd{B}([ B(x), \wtd{B}\circ B(y)])\right)\\
&+\lambda\left(B([\wtd{B}(x),\wtd{B}(y)]),\wtd{B}([B(x),B(y)])\right)-\left(B([ 
\wtd{B}\circ B(x),\wtd{B}(y)]), \wtd{B}([B\circ \wtd{B}(x),B(y)])\right ) \\
&+\left( B([\wtd{B}\circ B(y),\wtd{B}(x)]),\wtd{B}([B\circ \wtd{B}(y),B(x)])\right)\\
=&\lambda\left(B([\wtd{B}(x),\wtd{B}(y)]),\wtd
{B}([B(x),B(y)])\right).
\end{aligned}
\end{equation*}
This proves the assertion.
\end{proof}

\subsection{Rota-Baxter Lie algebras and Lie algebra projections on matched pairs of Lie algebras}

Now, let us introduce the notion of Lie algebra projections on matched pairs of Lie algebras.

\begin{defi}\label{FAC}
Let $(\frkg_+,\frkg_-,\rhd,\bhd)$ be a matched pair of Lie algebras. An operator $C:\frkg_+\bowtie \frkg
_-\to \frkg_+\bowtie \frkg_-$ is called a \textbf{Lie algebra projection on $(\frkg_+,\frkg_-,\rhd,\bhd)$} if  $C$ is an idempotent Lie algebra homomorphism.
\end{defi}
Let $(\frkg_+,\frkg_-,\rhd,\bhd)$ be a matched pair of Lie algebras and $C$ be a Lie algebra projection on $(\frkg_+,\frkg_-,\rhd,\bhd).$ In this paper, we denote $\wtd{C}=\id_{\frkg_+\bowtie\frkg_-}-C$. It follows directly from the definition that $\wtd{C}$ is a Lie algebra projection on $(\frkg_+,\frkg_-,\rhd,\bhd)$ and every Lie algebra projection on $(\frkg_+,\frkg_-,\rhd,\bhd)$ is a Rota-Baxter operator on $\frkg_+\bowtie \frkg_-$ of weight $-1$.

In the next proposition, we give an equivalent characterisation of Lie algebra projections on matched pairs of Lie algebras.

\begin{pro}
Let $(\frkg_+,\frkg_-,\rhd,\bhd)$ be a matched pair of Lie algebras. Then an operator $C:\frkg_+\bowtie \frkg_-\to \frkg_+\bowtie \frkg_-$ is a Lie algebra projection on $(\frkg_+,\frkg_-,\rhd,\bhd)$ if and only if there are Lie subalgebras $\frkg_1$ and $\frkg_2$ of $\frkg_+\bowtie \frkg_-$ such that $\frkg_+\bowtie \frkg_-=\frkg_1\oplus \frkg_2$ and $C$ is the projection from $\frkg$ to $\frkg_1.$ 
\end{pro}

In the next proposition, we show that every Lie algebra projection $C$ on a matched pair of Lie algebras $(\frkg_+,\frkg_-,\rhd,\bhd)$ induces a Rota-Baxter Lie algebra on $\ker C$.

\begin{pro}\label{MR}
Let $(\frkg_+,\frkg_-,\rhd,\bhd)$ be a matched pair of Lie algebras and $C:\frkg_+\bowtie \frkg_-\to \frkg_+\bowtie \frkg_-$ be a Lie algebra projection on $(\frkg_+,\frkg_-,\rhd,\bhd)$. Let $\frkg=\ker C$. Then the linear operators $B:\frkg\to \frkg$ defined by 
\begin{equation*}
B((x,u))=\wtd{C}((x,0)),\ \forall (x,u)\in \frkg,
\end{equation*}
and the operator 
$\wtd{B}:\frkg\to \frkg$ defined by 
\begin{equation*}
\wtd{B}((x,u))=\wtd{C}((0,u)),\ \forall (x,u)\in \frkg,
\end{equation*}
are Rota-Baxter operators of weight $-1$ on $\frkg$ such that 
$B+\wtd{B}=\id_{\frkg}.$
\end{pro}
\begin{proof}
Let $C$ be a Lie algebra projection on $(\frkg_+,\frkg_-,\rhd,\bhd).$ First, one can readily check that $B+\wtd{B}=\id_{\frkg} .$
For any $(x,u), (y,v)\in \frkg,$ we have 
\begin{equation*}
\begin{aligned}
&[B((x,u)),(y,v)]_{\bowtie}+[(x,u),B((y,v))]_{\bowtie}-[(x,u),(y,v)]_{\bowtie}\\
=&[B((x,u)),B((y,v))]_{\bowtie}+[B((x,u)),\wtd{B}((y,v))]_{\bowtie}+[(x,u),B((y,v))]_{\bowtie}-[(x,u),(y,v)]_{\bowtie}\\
=&[B((x,u)),B((y,v))]_{\bowtie}+[B((x,u)),\wtd{B}((y,v))]_{\bowtie}-[(x,u),\wtd{B}((y,v))]_{\bowtie}\\
=&[B((x,u)),B((y,v))]_{\bowtie}-[\wtd{B}((x,u)),\wtd{B}((y,v))]_{\bowtie}
\end{aligned}
\end{equation*}
Then we have 
\begin{equation*}
\begin{aligned}
&B\left([B((x,u)),(y,v)]_{\bowtie}+[(x,u),B((y,v))]_{\bowtie}-[(x,u),(y,v)]_{\bowtie}\right)\\
=&B\left([\wtd{C}((x,0)), \wtd{C}((y,0))]_{\bowtie}-[\wtd{C}((0,u)),\wtd{C}((0,v))]_{\bowtie}\right)\\
=&B\left(\wtd{C}([(x,0), (y,0)]_{\bowtie})-\wtd{C}([(0,u),(0,v)]_{\bowtie})\right)\\
=&B(([x,y],-[u,v])-C([(x,0),(y,0)]_{\bowtie})+C([(0,u),(0,v)]_{\bowtie}))\\ 
\end{aligned}
\end{equation*}
Finally, we have
\begin{equation*}
\begin{aligned}
&B\left([B((x,u)),(y,v)]_{\bowtie}+[(x,u),B((y,v))]_{\bowtie}-[(x,u),(y,v)]_{\bowtie}\right)\\
=&B(([x,y],-[u,v])-[C(-(0,u)),-C((0,v))]_{\bowtie}+[C((0,u)),C((0,v))]_{\bowtie})\\
=&B(([x,y],-[u,v]))=\wtd{C}(([x,y],0))=[\wtd{C}((x,0)),\wtd{C}((y,0))]_{\bowtie}=[B((x,u)),B((y,v))]_{\bowtie}.
\end{aligned}
\end{equation*}

Therefore, $B$ is a Rota-Baxter operator of weight $-1$ on $\frkg.$ It is similar to prove that $\wtd{B}$ is also a Rota-Baxter operator of weight $-1.$
\end{proof}

Here is the main result of this section. In the next theorem, we show that for any matched pair of Lie algebras on a Rota-Baxter Lie algebra of weight $-1,$ there are two Lie algebra projections on it such that the sum of the projections is the identity map. 
\begin{thm}\label{FL}
Let $(\frkg,B)$ be a Rota-Baxter Lie algebra of weight $-1$ and $(\frkg_+,\frkg_-,\rhd,\bhd)$ be the matched pair of Lie algebras on $(\frkg,B).$ Then the operator $C:\frkg_+\bowtie \frkg_-\to \frkg_+\bowtie \frkg_-$ given by 
\begin{equation*}
C\left((B(x_1),\wtd{B}(x_2))\right)=\left(B(B(x_1)+\wtd{B}(x_2)),\wtd{B}(B(x_1)+\wtd{B}(x_2))\right),\ \forall x_1,x_2\in\frkg,
\end{equation*}
and the operator $\wtd{C}:\frkg_+\bowtie \frkg_-\to G_+\bowtie G_-$ given by 
\begin{equation*}
\wtd{C}\left((B(x_1),\wtd{B}(x_2))\right)=\left(B\circ \wtd{B}(x_1-x_2),-\wtd{B}\circ B(x_1-x_2)\right),\ \forall x_1,x_2\in \frkg,
\end{equation*}
are Lie algebra projections on $(\frkg_+,\frkg_-,\rhd,\bhd)$ such that $C+\wtd{C}=\id_{\frkg_+\bowtie \frkg_-}.$
\end{thm}
\begin{proof}
First, let us prove $\wtd{C}$ is well defined. For any $x_1\in \ker B$ and $x_2\in \frkg,$ we have 
\begin{equation*}
\begin{aligned}
\wtd{C}\left( (B(x_1),\wtd{B}(x_2)) \right)&=\left(B\circ \wtd{B}(x_1-x_2),-\wtd{B}\circ B(x_1-x_2)\right)\\
&=\left(B\circ \wtd{B}(x_1)-B\circ\wtd{B}(x_2),-\wtd{B}\circ B(x_2)\right)\\
&=\left(\wtd{B}\circ B(x_1)-B\circ \wtd{B}(x_2),-\wtd{B}\circ B(x_2)\right)\\
&=\left(-B\circ \wtd{B}(x_2),-\wtd{B}\circ B(x_2)\right)\\
&=\wtd{C}\left(   (0,\wtd{B}(x_2))\right).
\end{aligned}
\end{equation*}
It is similar to prove
$$\wtd{C}\left( (B(x_1),\wtd{B}(x_2))\right)=\wtd{C}\left( (B(x_1),0)\right)$$ for any $x_1\in \frkg$ and $x_2\in \ker\wtd{B}$. This shows that $\wtd{C}$ is well defined. Then we verify that $C$ is a Lie algebra homomorphism. On the one hand, for any $x_1,x_2,y_1,y_2\in \frkg_+\bowtie \frkg_-,$ we have 
\begin{equation*} 
\begin{aligned}
&[(B(x_1),\wtd{B}(x_2)),(B(y_1),\wtd{B}(y_2))]_{\bowtie}\\=&\left([B(x_1),B(y_1)],[\wtd{B}(x_2),\wtd{B}(y_2)]\right)+\left(\wtd{B}(x_2)\bhd B(y_1),B(x_1)\rhd \wtd{B}(y_2) \right)-\left(\wtd{B}(y_2)\bhd B(x_1),B(y_1)\rhd \wtd{B}(x_2) \right)\\
=&\left([B(x_1),B(y_1)],[\wtd{B}(x_2),\wtd{B}(y_2)]\right)+\left(B([\wtd{B}(x_2),y_1]),\wtd{B}([B(x_1),y_2])\right)-\left( B([\wtd{B}(y_2),x_1]),\wtd{B}([B(y_1),x_2])\right)\\
=&\left(B([x_1,y_1]_{B}),\wtd{B}([x_2,y_2]_{\wtd{B}})\right)+\left(B([\wtd{B}(x_2),y_1]),\wtd{B}([B(x_1),y_2])\right)-\left( B([\wtd{B}(y_2),x_1]),\wtd{B}([B(y_1),x_2])\right)
\end{aligned}
\end{equation*}
It follows from \eqref{RBLRS}, \eqref{BTB} and Lemma \ref{RBRL} that 
\begin{equation*}
\begin{aligned}
&\wtd{C}\left([(B(x_1),\wtd{B}(x_2)),(B(y_1),\wtd{B}(y_2))]_{\bowtie}  \right)\\
=&\left(B\circ \wtd{B}([x_1,y_1]_{B}+[x_2,y_2]_{B}),-\wtd{B}\circ B([x_1,y_1]_{B}+[x_2,y_2]_{B}  \right)\\&+\left(B\circ \wtd{B}([\wtd{B}(x_2),y_1]-[B(x_1),y_2]),-\wtd{B}\circ B([\wtd{B}(x_2),y_1]-[B(x_1),y_2])  \right) \\
&-\left(B\circ \wtd{B}([\wtd{B}(y_2),x_1]-[B(y_1),x_2]),-\wtd{B}\circ B([\wtd{B}(y_2),x_1]-[B(y_1),x_2] ) \right) \\
=&\left(B\circ \wtd{B}([x_1,y_1]_{B}+[x_2,y_2]_{B}),-\wtd{B}\circ B([x_1,y_1]_{B}+[x_2,y_2]_{B})  \right)\\&-\left(B\circ \wtd{B}([x_2,y_1]_{B}),-\wtd{B}\circ B([x_2,y_1]_{B})\right)-\left(B\circ\wtd{B}([x_1,y_2]_{B}),-\wtd{B}\circ B([x_1,y_2]_{B}) \right)\\
=&\left(B\circ \wtd{B}([x_1-x_2,y_1-y_2]_{B}),-\wtd{B}\circ B([x_1-x_2,y_1-y_2]_{B}) \right)\\
=&-\left(B\circ [\wtd{B}(x_1-x_2),\wtd{B}(y_1-y_2)]),\wtd{B}\circ [B(x_1-x_2),B(y_1-y_2)]\right).
\end{aligned}
\end{equation*}
On the other hand, by Lemma \ref{BBW}, we have
\begin{equation*}
\begin{aligned}
&[\wtd{C}\left([(B(x_1),\wtd{B}(x_2))\right),C\left(B(y_1),\wtd{B}(y_2)\right)]_{\bowtie}\\
=&[\left(-B\circ \wtd{B}(x_1-x_2),\wtd{B}\circ B(x_1-x_2)\right),\left(-B\circ \wtd{B}(y_1-y_2),\wtd{B}\circ B(y_1-y_2)\right)]_{\bowtie}\\
=&-\left(B([\wtd{B}(x_1-x_2),\wtd{B}(y_1-y_2)]),\wtd{B}([B(x_1-x_2),B(y_1-y_2)] \right) .
\end{aligned}
\end{equation*}
This proves $\wtd{C}$ is an idempotent Lie algebra homomorphism.
Finally, it is straightforward to see that $C$ 
is also a Lie algebra projection on $(\frkg_+,\frkg_-,\rhd,\bhd)$. 
\end{proof}

Now we prove that every matched pair of Lie algebras on Rota-Baxter Lie algebras of weight $-1$ induces a Rota-Baxter Lie algebra of weight $-1$, which is Rota-Baxter isomorphic to a Rota-Baxter Lie algebra whose Rota-Baxter operator is idempotent. 

\begin{cor}\label{RBIS}
With the notations given in Theorem \ref{FL}, Corollary \ref{BBIS} and Theorem \ref{FN}, denote $\frkg_1=\im C$, $\frkg_2=\im \wtd{C}.$ Let $p_1:\frkg_1\oplus \frkg_2\to \frkg_1\bowtie \frkg_2$ be the projection from $\frkg_1\oplus \frkg_2$ to $\frkg_1$ and $p_2:\frkg_1\oplus \frkg_2\to \frkg_1\bowtie \frkg_2$ be the projection from $\frkg_1\oplus \frkg_2$ to $\frkg_2$. Then the following statements hold:
\begin{itemize}
\item[(a)]
$(\frkg_+\bowtie \frkg_-,C)$ is a Rota-Baxter Lie algebra of weight $-1$ that is Rota-Baxter isomorphic to $(\frkg_1\oplus \frkg_2,p_1);$ 
\item[(b)]
$(\frkg_+\bowtie \frkg_-,\wtd{C})$ is Rota-Baxter Lie algebra of weight $-1$ that is Rota-Baxter isomorphic to $(\frkg_1\oplus \frkg_2,p_2).$
\end{itemize}
\end{cor}

\begin{proof}
First, one can readily check that $(\frkg_+\bowtie \frkg_-,C)$ is a Rota-Baxter Lie algebra of weight $-1.$
Define the operator $\pi: (\frkg_+\bowtie \frkg_-,C)\to (\frkg_1\oplus \frkg_2,p_1)$ by
\begin{equation*}
\pi((B(x),\wtd{B}(y)))=\left(C\left((B(x),\wtd{B}(y))\right),\wtd{C}\left((B(x),\wtd{B}(y))\right)\right),\ \forall x,y\in \frkg.
\end{equation*}
It follows from Theorem \ref{FL} that $\pi$ is a Lie algebra homomorphism. Now we verify that $\pi$ is bijective. For any $x,y\in \frkg,$ if $\pi((B(x),\wtd{B}(y)))=(0,0),$ it follows from \ref{FL} that $$\left(B(x)),\wtd{B}(y)\right)=C\left((B(x),\wtd{B}(y))\right)+\wtd{C}\left((B(x),\wtd{B}(y))\right)=(0,0).$$ This proves that $\pi$ is injective. For any $(x_1,y_1), (x_2,y_2)\in \frkg_+\bowtie \frkg_-,$ by Theorem \ref{FL}, we have $$
\begin{aligned}
\pi(C((x_1,y_1))+\wtd{C}((x_2,y_2)))&=\left(C(C((x_1,y_1))+\wtd{C}((x_2,y_2))),\wtd{C}(C((x_1,y_1))+\wtd{C}((x_2,y_2)))\right)\\&=\left(C((x_1,y_1)),\wtd{C}((x_2,y_2))\right).
\end{aligned}
$$ This implies that $\pi$ is bijective. Then we show that $\pi$ is compatible with $C$ and $p_1.$ It follows directly from the definition that 
\begin{equation}\label{CPP}
C\circ\wtd{C}\left((B(x),\wtd{B}(y))\right)=\wtd{C}\circ C\left((B(x),\wtd{B}(y))\right)=(0,0), \forall x,y\in \frkg.
\end{equation}
Then for any $x,y\in\frkg,$  we have 
\begin{equation*}
\begin{aligned}
\pi\left(C\left((B(x),\wtd{B}(y))\right) \right)&=\left(C^{2}\left((B(x),\wtd{B}(y))\right),\wtd{C}\circ C\left(B(x),\wtd{B}(y)\right) \right)\\
&=\left( C\left((B(x),\wtd{B}(y))\right),0   \right)\ (\textit{by \eqref{CPP}})\\
&=p_1\left( C\left((B(x),\wtd{B}(y))\right),\wtd{C}\left((B(x),\wtd{B}(y))\right)\right)\\
&=p_1\circ \pi\left((B(x),\wtd{B}(y))\right).
\end{aligned}
\end{equation*}
Therefore, $(\frkg_+\bowtie \frkg_-,C)$ is Rota-Baxter isomorphic to $(\frkg_1\oplus \frkg_2,p_1)$ and it is similar to prove that $(\frkg_+\bowtie \frkg_-,\wtd{C})$ is Rota-Baxter isomorphic to $(\frkg_1\oplus \frkg_2,p_2).$
\end{proof}

In the next corollary, we prove that there is a Rota-Baxter Lie algebra of weight $-1$ on $\im \wtd{C}$ that is Rota-Baxter isomorphic to $(\frkg, B).$ 

\begin{cor}\label{BBIS}
Let $(\frkg,B)$ be a Rota-Baxter Lie algebra of weight $-1$ and $(\frkg_+,\frkg_-,\rhd,\bhd)$ be the matched pair of Lie algebras on $(\frkg,B).$ Let $C$ be the operator given in Theorem \ref{FL} and $\frkg_1=\im C.$ Define $B_1:\frkg_1\to \frkg_1$ and $\wtd{B_1}:\frkg_1\to\frkg_1$ by 
$$
\begin{aligned}
B_1((x,u))=&C((x,0)),\\
\wtd{B_1}((x,u))=&C((0,u)),
\end{aligned}
$$
for any $(x,u)\in \frkg_1.$
Then the following holds:
\begin{itemize}
\item[(a)]
$(\frkg_1,B_1)$ is a Rota-Baxter Lie algebra of weight $-1$ that is Rota-Baxter isomorphic to $(\frkg,B);$
\item[(b)]
$(\frkg_1,\wtd{B_1})$ is a Rota-Baxter Lie algebra of weight $-1$ that is Rota-Baxter isomorphic to $(\frkg,\wtd{B}).$
\end{itemize}
\end{cor}

\begin{proof}
By Proposition \ref{MR}, we know that $(\frkg_1,B_1)$ is a Rota-Baxter Lie algebra of weight $-1.$ Define the operator $\pi:(\frkg,B)\to (\frkg_1,B_1)$
\begin{equation*}
\pi(x)=(B(x),\wtd{B}(x)),\ \forall x\in \frkg.
\end{equation*}
It follows directly from the definition that $\pi$ is a bijection. Then by \eqref{MPR}, we get that $\pi$ is a Rota-Baxter Lie algebra homomorphism. Finally, we show that $B_1\circ \pi=\pi\circ B.$ For any $x\in \frkg,$ we have 
\begin{equation*}
B_1\circ \pi(x)=C((B(x),0)=(B\circ B(x),\wtd{B}\circ B(x))=(B\circ B(x),B\circ\wtd{B}(x))=\pi\circ B(x).
\end{equation*}
Therefore, $\pi$ is a Rota-Baxter Lie algebra isomorphism. This proves (a), it is similar to prove (b).
\end{proof}

 In the next proposition, we consider the image of $\wtd{C},$ where $\wtd{C}$ is defined in Theorem \ref{FL}.

\begin{pro}
Let $(\frkg,B)$ be a Rota-Baxter Lie algebra of weight $-1$ and $(\frkg_+,\frkg_-,\rhd,\bhd)$ be the matched pair of Lie algebras on $(\frkg,B).$ Let $\wtd{C}$ be the operator given in Theorem \ref{FL} and $\frkg_2=\im \wtd{C}.$ Then as a set, the following identity holds:
$$\frkg_2=\Big\{(x,-x)\in \frkg_+\bowtie \frkg_-|x\in\frkg_+\cap \frkg_-\Big\}.$$
\end{pro}
\begin{proof}
It follows directly from the definition of $B$ and $\wtd{B}$ that $\im B\circ \wtd{B}\subseteq \frkg_+\cap \frkg_-.$ For any $x,y\in \frkg,$ if $B(x)=\wtd{B}(y),$ then we have $x=B(x)+\wtd{B}(x)=\wtd{B}(x+y).$ It follows that $B(x)=B\circ \wtd{B}(x+y).$ This implies $ \frkg_+\cap \frkg_-\subseteq \im B\circ \wtd{B}.$ Therefore, as sets, we obtain that $\im B\circ \wtd{B}= \frkg_+\cap \frkg_-$ and 
$$\frkg_2=\Big\{(x,-x)\in \frkg_+\bowtie \frkg_-\big |x\in\frkg_+\cap \frkg_-\Big\}.$$ 
\end{proof}

The next Lemma will be used in the proof of Proposition \ref{OB}.

\begin{lem}\label{B1}
Let $(\frkg,B)$ be a Rota-Baxter Lie algebra of weight $\lambda.$ Then the operator $B:\frkg_B\to \frkg_B$ 
is a Rota-Baxter operator of weight $\lambda$ on $\frkg_B.$
\end{lem}

\begin{proof}
For any $x,\ y\in \frkg,$ we have
\begin{equation*}
\begin{aligned}
\relax [B(x),B(y)]_{B}&=[B\circ B(x),B(y)]+[B(x),B\circ B(y)]+\lambda [B(x),B(y)]\\
&=B([B(x),y]_{B}+[x,B(y)]_{B}+\lambda [x,y]_{B})\ (\textit{by \eqref{RBLRS}}).
\end{aligned}
\end{equation*}
This proves $B$ is a Rota-Baxter operator on $\frkg_B$ of weight $\lambda.$
\end{proof}
The Rota-Baxter Lie algebra $(\frkg_B,B)$ given in the above lemma is called the descendent Rota-Baxter Lie algebra of $(\frkg,B).$
\begin{lem}
Let $(\frkg,B)$ be a Rota-Baxter Lie algebra of weight $\lambda.$ Let $\frkh_+ + \frkh_-$ be the Lie subalgebra generated by $\frkh_+$ and $\frkh_-.$ Then $\frkh_+ + \frkh_-$ is an ideal of $\frkg_B$ and an ideal of $\frkg_{\wtd{B}}.$
\end{lem}
\begin{proof}
By \eqref{RBLRS}, $\frkh_+$ and $\frkh_-$ are both ideals of $\frkg_B$ and $\frkg_{\wtd{B}}.$ Therefore $\frkh_+ + \frkh_-$ is an ideal of $\frkg_B$ and an ideal of $\frkg_{\wtd{B}}.$
\end{proof}

Next, we prove that every Rota-Baxter Lie algebra of weight $-1$ $(\frkg,B)$ induces a Rota-Baxter Lie algebra structure on $\frkg_B /(\frkh_++\frkh_-).$
\begin{pro}\label{OB}
Let $(\frkg,B)$ be a Rota-Baxter Lie algebra of weight $-1$. Define the operator $\overline{B}:\frkg_{B}/(\frkh_+ + \frkh_-) \to \frkg_{B}/(\frkh_+ + \frkh_-)$ by 
\begin{equation*}
\overline{B}(\overline{x})=\overline{B(x)},\ \forall x\in \frkg,
\end{equation*}
and the operator $\wtd{\ol{B}}:\frkg_{\wtd{B}}/(\frkh_+ + \frkh_-) \to \frkg_{\wtd{B}}/(\frkh_+ + \frkh_-)$ by 
\begin{equation*}
\wtd{\ol{B}}(\overline{x})=\ol{\wtd{B}(x)},\ \forall x\in \frkg.
\end{equation*}
Then $\overline{B}$ and $\wtd{\ol{B}}$ are both Rota-Baxter operators of weight $-1$ on $\frkg$.
\end{pro}

\begin{proof}
First, let us prove that $\overline{B}$ is well defined. For any $x\in \frkh_+,$ and $u\in \frkh_-,$ we have
$B(x+u)=B(x)=x\in \frkh_+ + \frkh_-.$ This shows that $\overline{B}$ is well defined. Then by Lemma \ref{B1}, we know that $\overline{B}$ is a Rota-Baxter operator of weight $-1$ on $\frkg_B/(\frkh_+ +\frkh_-).$ It is similar to prove that $\wtd{\ol{B}}$ is also a Rota-Baxter operator of weight $-1$ on $\frkg_B/(\frkh_+ +\frkh_-).$ Finally, one can readily verify that $\overline{B}+\wtd{\ol{B}}=\id_{\frkg_{B}/(\frkh_+ + \frkh_-)}$ from the definition of $\overline{B}$ and $\wtd{\ol{B}}.$
\end{proof}

In Corollary \ref{BBIS}, we have investigated the Rota-Baxter Lie algebra on $\im C$. Then we study the Rota-Baxter Lie algebra on $\im \wtd{C},$ and prove that it is Rota-Baxter isomorphic to the Rota-Baxter Lie algebra defined in Proposition \ref{OB}.

\begin{thm}\label{FN}
Let $(\frkg,B)$ be a Rota-Baxter Lie algebra and $(\frkg_+,\frkg_-,\rhd,\bhd)$ be the matched pair of Lie algebras on $(\frkg,B).$ Let $C$ and $\wtd{C}$ be the operators given in Theorem \ref{FL}, $\overline{B}$ be the operator given in Proposition \ref{OB} and $\frkg_2=\im \wtd{C}$. Define $B_2:\frkg_2\to \frkg_2$ by 
\begin{equation*}
B_2((x,u))=\wtd{C}((0,u)),\ \forall (x,u)\in \frkg_2,
\end{equation*}
and
$\wtd{B_2}:\frkg_2\to \frkg_2$ by 
\begin{equation*}
\wtd{B_2}((x,u))=\wtd{C}((x,0)),\ \forall (x,u)\in \frkg_2.
\end{equation*}
Then the following holds:
\begin{itemize}
\item[(a)]
$(\frkg_2,B_2)$ is a Rota-Baxter Lie algebra of weight $-1$ that is Rota-Baxter isomorphic to $(\frkg_{B}/(\frkh_+ +\frkh_-),\overline{B});$
\item[(b)] $(\frkg_2,\wtd{B_2})$ is a Rota-Baxter Lie algebra of weight $-1$ that is Rota-Baxter isomorphic to $(\frkg_{\wtd{B}}/(\frkh_+ +\frkh_-),\wtd{\ol{B}});$
\end{itemize}
\end{thm}
\begin{proof}
By Theorem \ref{FL} and Proposition \ref{MR}, we know that $B_2$ is a Rota-Baxter operator of weight $-1$ on $\frkg_2.$ Define the operator $\pi:(\frkg_{B}/(\frkh_+ +\frkh_-),\ol{B}) \to (\frkg_2,B_2)$ by 
\begin{equation*}
\pi(\overline{x})=(B\circ \wtd{B}(x),-B\circ \wtd{B}(x)), \ \forall x\in \frkg_{B}.
\end{equation*}
First, it follows directly from the definition that $\pi$ is well defined. Next, we check that $\pi$ is a Lie algebra homomorphism. 
For any $x,\ y\in \frkg,$ we have 
\begin{equation*}
\begin{aligned}
[\pi(\ol{x}),\pi(\ol{y})]_{\bowtie}&=[(B\circ \wtd{B}(x),-B\circ \wtd{B}(x)),(B\circ \wtd{B}(y),-B\circ \wtd{B}(y))]_{\bowtie}\\
&=\left(-B([\wtd{B}(x),\wtd{B}(y)]),-\wtd{B}([B(x),B(y)])\right)\ (\textit{by Lemma \ref{BBW}})\\
&=\left(B\circ \wtd{B}([x,y]_{B}),-B\circ \wtd{B}([x,y]_{B})\right)\ (\textit{by \eqref{RBLRS} and Lemma \ref{RBRL}})\\ 
&=\pi(\ol{[x,y]_{B}}).
\end{aligned}
\end{equation*}
This proves $\pi$ is a Lie algebra homomorphism. For any $x\in \frkg,$ if 
$\pi(\ol{x})=(0,0).$ It follows that $\wtd{B}(x)\in \frkh_-$ and $B(x)\in \frkh_+.$ Then we have $x=B(x)+\wtd{B}(x)\in \frkh_+ +\frkh_-.$ This shows that $\pi$ is injective. By the definition of $\pi$, it is not hard to see that $\pi$ is surjective. Finally, we prove that $\pi\circ \ol{B}=B_2\circ \pi.$ For any $x\in \frkg,$
we have 

\begin{equation*}
\begin{aligned}
\pi\circ \ol{B}(\ol{x})&=\pi(\ol{B(x)})=\left(B\circ\wtd{B}\circ B(x),-B\circ\wtd{B}\circ B(x)\right)\\&=C\left((0,-\wtd{B}\circ B(x) )\right)=B_2\circ \pi(\ol{x}).
\end{aligned}
\end{equation*}
Therefore, $(\frkg_2,B_2)$ is Rota-Baxter isomorphic to $(\frkg_{B}/(\frkh_+ +\frkh_-),\overline{B}).$ This proves (a) and it is to prove (b).
\end{proof}

By the above theorem, one can readily obtain the following corollary. 
\begin{cor}
With the notations given in \ref{FN}, define $\pi_1:(\frkg_{B},B) \to (\frkg_2,B_2)$ and $\pi_2:(\frkg_{\wtd{B}},\wtd{B}) \to (\frkg_2,\wtd{B_2})$
by 
\begin{equation*}
\begin{aligned}
\pi_1(x)&=(B\circ \wtd{B}(x),-B\circ \wtd{B}(x)),\\
\pi_2(x)&=(B\circ \wtd{B}(x),-B\circ \wtd{B}(x)),
\end{aligned}
\end{equation*}
for any $x\in \frkg.$ Then $\pi_1$ and $\pi_2$ are surjective Rota-Baxter Lie algebra homomorphisms.
\end{cor}

\section{Manin triples and quadratic Rota-Baxter Lie algebras}\label{3}

In this section, we consider the relationship between Manin triples and quadratic Rota-Baxter Lie algebras, the latter introduced in \cite{Hl}. We show that every quadratic Rota-Baxter Lie algebra of weight $-1$ gives rise to the structure of a Manin triple. Moreover, we prove that the induced Manin triple admits a decomposition.

First, let us recall the definition of quadratic Lie algebras.
Let $\frkg$ be a Lie algebra. Recall that a nondegenerate symmetric bilinear form $S:\frkg\times\frkg \to \mathbb{K}$ is called invariant if it satisfies 
\begin{equation*}
S([x,y],z)=S(x,[y,z]),\ \forall x,y,z\in \frkg.
\end{equation*}
 Moreover, the double $(\frkg,S)$ is called a quadratic Lie algebra.

Next, let us recall the definition of Manin tripes.

A Manin triple is a triple $((\frkg, S),\frkg_{+},\frkg_{-})$ where $(\frkg, S)$ is a quadratic Lie algebra, $\frkg_{+}$ and $\frkg_-$ are Lie algebras such that 

\begin{itemize}
\item[(a)] $\frkg_{+}$ and $\frkg_{-}$
are Lie subalgebras of $\frkg;$
\item[(b)] $\frkg=\frkg_{+}\oplus \frkg_{-}$ as vector spaces;
\item[(c)] $\frkg_{+}$ and $\frkg_{-}$ are isotropic with respect to the bilinear form $S$.
\end{itemize}

Let $((\frkg,S),\frkg_{+},\frkg_{-})$ be a Manin triple. As $\frkg$ is a Lie algebra, then by Remark \ref{MM}, we know that the Lie bracket of $\frkg$ induces a matched pair of Lie algebra structure $(\frkg_+,\frkg_-,\rhd,\bhd).$ From this to the end of this paper, we call it \textbf{the matched pair of Lie algebras on the Manin triple $((\frkg,S),\frkg_{+},\frkg_{-}).$}

Then let us recall the notion of quadratic Rota-Baxter Lie algebras.

 \begin{defi}
A triple $(\mathfrak{g}, B, S)$ is called a \textbf{quadratic Rota-Baxter Lie algebra} of weight $\lambda$ if $(\frkg, B)$ is a Rota-Baxter Lie algebra of weight $\lambda,$ $(\frkg, S)$ is a quadratic Lie algebra, and the following condition holds:
\begin{equation}\label{RSP}
S(B(x),y)+S(x,B(y))+\lambda S(x,y)=0,\ \forall x,y\in \frkg.
\end{equation}
 \end{defi}

Note that \eqref{RSP} can be also written as  
\begin{equation}\label{RP}
S(B(x),y)=S(x,\wtd{B}(y)),\ \forall x,y\in \frkg.
\end{equation}

The next lemma will be used in the proof of Proposition \ref{MRS}.

\begin{lem}\label{IV1}
Let $(\frkg,B,S)$ be a quadratic Rota-Baxter Lie algebra of weight $\lambda.$ Then the operator $S':(\frkg_+\oplus \frkg_-)\times (\frkg_+\oplus \frkg_-)\to \frkg_+\oplus \frkg_-$ given by 
\begin{equation}\label{DS'}
S'\left((B(x_1),\wtd{B}(x_2)),(B(y_1),\wtd{B}(y_2))\right)=S(B(x_1),y_2)+S(B(y_1)  ,x_2),\ \forall x_1,x_2,y_1,y_2\in \frkg
\end{equation}
is a nondegenerate symmetric invariant bilinear form.
\end{lem}
\begin{proof}
First, let us prove that $S'$ is well defined. For any $x_1,y_1\in \frkg$ and $x_2,y_2\in \ker \wtd{B},$ by \eqref{RP} we have
\begin{equation*}
S(B(x_1),y_2)+S(B(y_1),x_2)=S(x_1,\wtd{B}(y_2))+S(y_1,\wtd{B}(x_2))=0.
\end{equation*}
This proves $S'$ is well defined. From the definition of $S'$, it follows that $S'$ is symmetric. Next, we prove that $S'$ is nondegenerate. Let $x_1,x_2\in\frkg.$ If for any $y_1,y_2\in \frkg,$ 
\begin{equation*}
S'\left(B(x_1),\wtd{B}(x_2)),(B(y_1),\wtd{B}(y_2))\right)=S(B(x_1),y_2)+S(B(y_1),x_2)=0,
\end{equation*}
then we have $S(B(x_1),y_2)=0$. Again by \eqref{RP}, we have $S(y_1,\wtd{B}(x_2)))=S(B(y_1),x_2)=0$. As $S$ is nondegenerate, we have $(B(x_1),\wtd{B}(x_2))=(0,0).$ It follows that $S'$ is nondegenerate. Finally, we prove that $S'$ is invariant. For any $x_1,x_2,y_1,y_2,z_1,z_2\in \frkg,$ we have
\begin{equation*}
\begin{aligned}
&S'\left([(B(x_1),\wtd{B}(x_2)),(B(y_1),\wtd{B}(y_2))]_{\bowtie}, (B(z_1),\wtd{B}(z_2))\right)\\
=&S'\left([(B(x_1),B(y_1)]+\wtd{B}(x_2)\bhd B(y_1)-\wtd{B}(y_2)\bhd B(x_1), \wtd{B}(z_2)\right)\\
&+S'\left( [\wtd{B}(x_2),\wtd{B}(y_2)]+B(x_1)\rhd \wtd{B}(y_2)-B(y_1)\rhd \wtd{B}(x_2),B(z_1)\right)\\
=&S'\left([(B(x_1),B(y_1)]+B([\wtd{B}(x_2),y_1])-B([\wtd{B}(y_2),x_1]), \wtd{B}(z_2)\right)\\
&+S'\left( [\wtd{B}(x_2),\wtd{B}(y_2)]+\wtd{B}([B(x_1),y_2])-\wtd{B}([B(y_1),x_2]),B(z_1)\right)\\
=&S\left([(B(x_1),B(y_1)], z_2\right)+S\left([\wtd{B}(x_2),y_1], \wtd{B}(z_2)\right)+S\left([x_1,\wtd{B}(y_2)], \wtd{B}(z_2)\right)\\
&+S\left( [\wtd{B}(x_2),\wtd{B}(y_2)],z_1\right)+S\left( [B(x_1),y_2],B(z_1)\right)+S\left( [x_2,B(y_1)],B(z_1)\right) \\
\end{aligned}
\end{equation*}
Then as $S$ is invariant, we have 
\begin{equation*}
\begin{aligned}
&S'\left([(B(x_1),\wtd{B}(x_2)),(B(y_1),\wtd{B}(y_2))]_{\bowtie}, (B(z_1),\wtd{B}(z_2))\right)\\
=&S\left(B(x_1),[B(y_1), z_2]\right)-S\left(B(x_1),[B(z_1),y_2]\right)+S\left(x_1,[\wtd{B}(y_2),\wtd{B}(z_2)]\right)\\
&-S\left(\wtd{B}(x_2), [\wtd{B}(z_2),y_1]\right)+S\left( \wtd{B}(x_2),[\wtd{B}(y_2),z_1]\right)+S\left( x_2,[B(y_1),B(z_1)] \right)\\
=&S\left(B(x_1),[B(y_1), z_2]\right)-S\left(B(x_1),[B(z_1),y_2]\right)+S\left(B(x_1),[y_2,z_2]_{\wtd{B}}\right)\\
&-S\left(\wtd{B}(x_2), [\wtd{B}(z_2),y_1]\right)+S\left( \wtd{B}(x_2),[\wtd{B}(y_2),z_1]\right)+S\left(B(x_2),[y_1,z_1]_{B} \right)\ (\textit{by \eqref{RBLRS} and \eqref{RP}})\\
=&S'\left((B(x_1),0),(0,[\wtd{B}(y_2),\wtd{B}(z_2)]+B(y_1)\rhd \wtd{B}(z_2)-B(z_1)\rhd \wtd{B}(y_2)   )  \right)\\
&+S'\left((0,\wtd{B}(x_2)),([B(y_1),B(z_1)]+\wtd{B}(y_2)\bhd B(z_1)-\wtd{B}(z_2)\bhd B(y_1),0)  \right)\ (\textit{by \eqref{RP} and \eqref{DS'}})\\
=&S'\left((B(x_1),\wtd{B}(x_2)),[(B(y_1),\wtd{B}(y_2)),(B(z_1),\wtd{B}(z_2))]_{\bowtie}  \right).
\end{aligned}
\end{equation*}
This proves $S'$ is invariant.

\end{proof}

In the next theorem, we show that every quadratic Rota-Baxter Lie algebra of weight $-1$ $(\frkg,B,S)$ induces a Manin triple structure on $\frkg_+\oplus \frkg_-.$

\begin{pro}\label{MRS}
With the notations given above. Let $(\frkg,B,S)$ be a quadratic Rota-Baxter Lie algebra of weight $-1$ and $(\frkg_+,\frkg_-,\rhd,\bhd)$ be the matched pair of Lie algebras on $(\frkg,B)$. Then $((\frkg_+\bowtie \frkg_-,S'),\frkg_+,\frkg_-)$ is a Manin triple.
\end{pro}
\begin{proof}
$\frkg_+$ and $\frkg_-$ are both Lie subalgebras of $\frkg\oplus \frkg_-$ with respect to $[\cdot,\cdot]_{\bowtie}$. It follows from Lemma \ref{IV1} that $S'$ is a nondegenerate symmetric invariant bilinear form. Therefore $(\frkg_+\bowtie \frkg_-,S')$ is a quadratic Lie algebra. Finally, it follows from the definition of $S'$ that $\frkg_+$ and $\frkg_-$ are isotropic with respect to $S'.$ This proves that $((\frkg_+\bowtie \frkg_-,S'),\frkg_+,\frkg_-)$ is a Manin triple.
\end{proof} 

This matched pair given in the above Theorem is called \textbf{the matched pair of Lie algebras on the quadratic Rota-Baxter Lie algebra $(\frkg,B,S).$}

\begin{rmk}
By \cite{K}, a factorizable Lie bialgebra is a quasitriangular Lie bialgebra $(\mathfrak g,r)$ whose $r$-matrix has a nondegenerate symmetric part.  In other words, a factorizable Lie bialgebra is the one for which $r$ solves the classical Yang--Baxter equation and its symmetric part $s$ defines a nondegenerate, invariant bilinear form on $\mathfrak g$. This nondegeneracy makes the maps
\[
r_\pm:\mathfrak g^*\to\mathfrak g,\qquad
r_\pm(\xi)=\langle \xi,\cdot\rangle_{\text{first/second component of } r}
\]
isomorphisms of Lie algebras onto their images. For more details on factorizable Lie bialgebras, we refer to \cite{Hl,K}.

In \cite{Hl}, it was shown that there is a
one-to-one correspondence between
factorizable Lie bialgebras and quadratic Rota-Baxter Lie algebras of nonzero weight. Therefore, Proposition \ref{MRS} implies that for any factorizable Lie bialgebra $(\frkg,r),$ there is a matched pair of Lie algebras on $\im r_+\oplus \im r_-.$ 
\end{rmk}

Then we introduce the definition of quadratic Lie algebra projections on Manin triples.

\begin{defi}\label{FM}
Let $((\frkg, S),\frkg_+,\frkg_-)$ be a Manin triple and $(\frkg_+,\frkg_-,\rhd,\bhd)$ be the matched pair of Lie algebras on $((\frkg, S),\frkg_+,\frkg_-)$. Let $C:\frkg\to \frkg$ be a Lie algebra projection on $(\frkg_+,\frkg_-,\rhd,\bhd)$. Let $\wtd{C}=\id_{\frkg}-C.$
We call $C$ a \textbf{quadratic Lie algebra projection on the Manin triple} $((\frkg,S),\frkg_+,\frkg_-)$ if it satisfies
 $$S(C((x_1,x_2)),\wtd{C}((y_1,y_2)))=0, \forall (x_1,x_2),\ (y_1,y_2)\in \frkg.$$ 
\end{defi}

Next, as the quadratic Lie algebra analogue of Rota-Baxter Lie algebra homomorphisms, we propose the notion of quadratic Rota-Baxter Lie algebra homomorphisms.

\begin{defi}
Let $(\frkg,B,S)$ and $(\frkg',B',S')$ be quadratic Rota-Baxter Lie algebras. A map $f:(\frkg,B,S)\to(\frkg',B',S')$ is called a \textbf{quadratic Rota-Baxter Lie algebra homomorphism} if $f$ is a Rota-Baxter Lie algebra homomorphism and it satisfies 
$$S'(f(x),f(y))=S(x,y),\ \forall x,\ y\in \frkg.$$
Moreover, if $f$ is bijective, then $f$ is called a \textbf{quadratic Rota-Baxter Lie algebra isomorphism} and $(\frkg,B,S)$ is called \textbf{quadratic Rota-Baxter isomorphic} to $(\frkg',B',S')$.
\end{defi}

To study the decomposition of the Manin triple on a quadratic Rota-Baxter Lie algebra of weight $-1,$ we introduce the notion of the direct sum of quadratic Lie algebras. 
\begin{defi}
Let $(\frkg_1,S_1)$ and $(\frkg_2,S_2)$ be quadratic Lie algebras. The pair $(\frkg,S)$ is called the direct sum of quadratic Lie algebras $(\frkg_1,S_1)$ and $(\frkg_2,S_2)$ if $\frkg=\frkg_1\oplus \frkg_2$ and $S:\frkg\times \frkg\to \frkg$ is given by 
$$S((x,u),(x',u'))=S_1(x,x')+S_2(u,u'), \ \forall x,x'\in \frkg_1, \ \forall u,u'\in \frkg_2.$$
\end{defi}

It is straightforward to obtain the following proposition from the definition of quadratic Lie projections on Manin triples.

\begin{pro}
Let $((\frkg,S),\frkg_+,\frkg_-)$ be a Manin triple. Then an operator $C:\frkg\to \frkg$ is a quadratic Lie algebra projection on $((\frkg,S),\frkg_+,\frkg_-)$ if and only if there are quadratic Lie  subalgebras $(\frkg_1,S_1) \subseteq (\frkg,B)$ and $(\frkg_2,S_2)\subseteq (\frkg,B)$ such that $(\frkg,S)=(\frkg_1,S_1)\oplus (\frkg_2,S_2)$ and $C$ is the projection from $\frkg$ to $\frkg_1.$

\end{pro}

Now, we give the Manin triple version of Proposition \ref{MR}.
\begin{pro}
Let $((\frkg,S),\frkg_+,\frkg_-)$ be a Manin triple and $C$ be a Lie algebra projection on $((\frkg,S),\frkg_+,\frkg_-)$. Let $\wtd{C}=\id_{\frkg}-C$ and $\frkh=\ker C.$ Then the following statements hold:
\begin{itemize}
\item[(a)] $(\frkh,B,S_{\frkh})$ is a quadratic Rota-Baxter Lie algebra of weight $-1,$ where $B:\frkh \to \frkh$ is defined by 
\begin{equation*}
B((x,u))=\wtd{C}((x,0)),\ \forall (x,u)\in \frkh,
\end{equation*}
and $S_{\frkh}$ is the restriction of $S$ to $\frkh\times \frkh;$
\item[(b)] $(\frkh,\wtd{B},S_{\frkh})$ is a quadratic Rota-Baxter Lie algebra of weight $-1,$ where $\wtd{B}:\frkh \to \frkh$ is defined by 
\begin{equation*}
B((x,u))=\wtd{C}((0,u)),\ \forall (x,u)\in \frkh,
\end{equation*}
and $S_{\frkh}$ is the restriction of $S$ to $\frkh\times\frkh.$
\end{itemize}
\end{pro}

\begin{proof}
By Proposition \ref{MR}, we know that $B$ and $\wtd{B}$ are Rota-Baxter operators of $\frkh$ such that $B+\wtd{B}=\id_{h}.$ As $S$ is symmetric and invariant, we know that $S_{\frkh}$ is also symmetric and invariant. Next, we show that $S_{\frkh}$ is nondegenerate. Let $(x,u)\in \frkg.$ If for any $(y,u)\in \frkg,$ $S_{\frkh}(\wtd{C}((x,u)),\wtd{C}((y,v)))=0,$ from the definition of $S_{\frkh},$ it follows that 
\begin{equation*}
\begin{aligned}
S(\wtd{C}((x,u)),(y,v))=&S(\wtd{C}((x,u)),C((y,v))+\wtd{C}((y,v)))\\
=&S(\wtd{C}((x,u)),C((y,v)))+S(\wtd{C}((x,u)),\wtd{C}((y,v)))\\
=&S(\wtd{C}((x,u)),\wtd{C}((y,v)))=0.
\end{aligned}
\end{equation*}
As $S$ is nondegenerate, we have $\wtd{C}((x,u))=0.$ This proves $S_{\frkh}$ is nondegenerate. Finally, we prove that \eqref{RP} holds. For any $(x,u),(y,v)\in \frkh,$ we have 
\begin{equation*} 
\begin{aligned}
S_{\frkh}\left(B((x,u)),(y,v)\right)&=S\left(\wtd{C}(x,0),\wtd{C}((y,v))\right)\\
&=S((x,0),(y,v))-S(C(x,0),C((y,v)))\\
&=S((x,0),(0,v)).\\
\end{aligned}
\end{equation*}
On the other hand, it is similar to prove that $$S_{\frkh}\left((x,u),\wtd{B}((y,v))\right)=S_{\frkh}\left(\wtd{C}((x,u)),\wtd{C}((0,v)\right)=S\left((x,0),(0,v))\right).$$
This implies $S_{\frkh}\left(B((x,u)),(y,v)\right)=S_{\frkh}\left((x,u),\wtd{B}((y,v))\right).$ Hence $(\frkh,B,S_{\frkh})$ is a quadratic Rota-Baxter Lie algebra of weight $-1$ and it is similar to prove that $(\frkh,\wtd{B},S_{\frkh})$ is a quadratic Rota-Baxter Lie algebra of weight $-1.$
\end{proof} 

Then we give the quadratic Rota-Baxter Lie algebra version of Theorem \ref{FL}.

\begin{thm}\label{QRB}
Let $(\frkg,B,S)$ be a quadratic Rota-Baxter Lie algebra of weight $-1$ and $((\frkg_+\bowtie \frkg_-,S'),\frkg_+,\frkg_-)$ be the Manin triple on $(\frkg,B,S).$ Then the operator $C:\frkg_+\bowtie \frkg_-\to \frkg_+\bowtie \frkg_-$ given by 
\begin{equation*}
C\left((B(x),\wtd{B}(y))\right)=\left(B(B(x)+\wtd{B}(y),\wtd{B}(B(x)+\wtd{B}(y)))\right),\ \forall x,y\in\frkg,
\end{equation*}
and the operator $\wtd{C}:\frkg_+\bowtie \frkg_-\to \frkg_+\bowtie \frkg_-$ given by 
\begin{equation*}
\wtd{C}\left((B(x_1),\wtd{B}(x_2))\right)=\left(B\circ \wtd{B}(x_1-x_2),-\wtd{B}\circ B(x_1-x_2)\right),\ \forall x_1,x_2\in \frkg,
\end{equation*}
are Lie algebra projections on $((\frkg_-,S'),\frkg_+,\frkg_-)$ such that $C+\wtd{C}=\id_{\frkg_+\bowtie \frkg_-}.$
\end{thm}

\begin{proof}
Let $(\frkg_+,\frkg_-,\rhd,\bhd)$ be the Lie algebra projection on $(\frkg,B).$ By Theorem \ref{FL}, we know that $C$ is a Lie algebra projection on $(\frkg_+,\frkg_-,\rhd,\bhd).$ For any $x_1,\ y_1,\ x_2,\ y_2\in \frkg,$ we have 
\begin{equation*}
\begin{aligned}
&S'\left( \wtd{C}(B(x_1),\wtd{B}(x_2)), C((B(y_1),\wtd{B}(y_2)))\right)\\
=&S'\left((B\circ\wtd{B}(x_1-x_2), -\wtd{B}\circ B(x_1-x_2)), (B(B(y_1)+\wtd{B}(y_2)),\wtd{B}(B(y_1)+\wtd{B}(y_2)))\right)\\
=&S\left(B\circ\wtd{B}(x_1-x_2), B(y_1)+\wtd{B}(y_2)\right)+S\left(B(B(y_1)+\wtd{B}(y_2)),-B(x_1-x_2)\right)\\
=&S\left(B(x_1-x_2), B(B(y_1)+\wtd{B}(y_2))\right)+S\left(-B(x_1-x_2),B(B(y_1)+\wtd{B}(y_2))\right)\\
=&0.
\end{aligned}
\end{equation*}
Therefore, $C$ is a quadratic Lie algebra projection on $((\frkg_+\bowtie \frkg_-,S'),\frkg_+,\frkg_-).$ It is similar to prove that $\wtd{C}$ is a quadratic Lie algebra projection on $((\frkg_+\bowtie \frkg_-,S'),\frkg_+,\frkg_-).$
\end{proof}
Then, by Corollary \ref{RBIS} and Theorem \ref{QRB}, it is not difficult to obtain the following corollary.
\begin{cor}\label{ABV}
With the notations given in Theorem \ref{QRB}, let $\frkg_1=\im C$, $\frkg_2=\im \wtd{C}$, $p_1:\frkg_1\oplus \frkg_2\to \frkg_1\bowtie \frkg_2$ and $p_2:\frkg_1\oplus \frkg_2\to \frkg_1\bowtie \frkg_2$ be the projection from $\frkg_1\oplus \frkg_2$ to $\frkg_1$ and $\frkg_2$ respectively. Then the following holds:
\begin{itemize}
\item[(a)] $(\frkg_+\bowtie \frkg_-,S',C)$ is a quadratic Rota-Baxter Lie algebra that is isomorphic to $(\frkg_1\oplus \frkg_2,S_{\frkg_1}\oplus S_{\frkg_2},p_1);$
\item[(b)] 
$(\frkg_+\bowtie \frkg_-,S',\wtd{C})$ is a quadratic Rota-Baxter Lie algebra that is isomorphic to $(\frkg_1\oplus \frkg_2,S_{\frkg_1}\oplus S_{\frkg_2},p_2).$
\end{itemize}
\end{cor}

Finally, by Corollary \ref{BBIS} and Theorem \ref{QRB}, we get the following corollary.

\begin{cor}\label{RIS}
With the notations given in Theorem \ref{QRB} and Corollary \ref{ABV}. Define $B_1:\frkg_1\to \frkg_1$ and $\wtd{B_1}:\frkg_1\to\frkg_1$ by 
$$
\begin{aligned}
B_1((x,u))=&C((x,0)),\\
\wtd{B_1}((x,u))=&C((0,u)),
\end{aligned}
$$
for any $(x,u)\in \frkg_1.$
Then the following holds: 
\begin{itemize}
\item[(a)] $(\frkg_1,S_{\frkg_1},B_1)$ is a quadratic Rota-Baxter Lie algebra of weight $-1$ that is quadratic Rota-Baxter isomorphic to $(\frkg,S,B)$;
\item[(b)] $(\frkg_1,S_{\frkg_1},\wtd{B_1})$ is a quadratic Rota-Baxter Lie algebra of weight $-1$ that is quadratic Rota-Baxter isomorphic to $(\frkg,S,\wtd{B})$.
\end{itemize}
\end{cor}

\section{Matched pairs of groups and Rota-Baxter groups}\label{4}
In this section, we first recall some basic notions of Rota-Baxter groups $(G,\hB)$ of weight $-1$ and prove that every Rota-Baxter group of weight $-1$ induces a matched pair of groups, which we call the matched pair of groups on $(G,\hB)$. Then propose the notion of group projections on a matched pair of groups, investigating its relationship with the matched pair of groups on $(G,\hB)$. Finally, we study the decomposition of the induced matched pair of groups.

Let $G$ be a group. In this section, we denote the identity element of $G$ by $e.$

\subsection{Preliminaries on matched pairs of groups and Rota–Baxter groups}

To define the matched pairs of groups on a Rota-Baxter group of weight $-1$, firstly, let us recall the notion of Rota-Baxter groups of weight $-1$ introduced in \cite{LG}.

\begin{defi}
Let $G$ be a group. Then an operator $\hB:G\to G$ is called a Rota-Baxter operator of $-1$ on $G$ if it satisfies 
\begin{equation}\label{RBLS}
\hB(a)\hB(b)=\hB((\Ad_{\hB(a)}(b))a),\ \forall a\in G.
\end{equation}
The pair $(G,\hB)$ is called a Rota-Baxter Lie group of weight $-1$. Moreover, if $G$ is a Lie group and $\hB$ is a smooth operator, then $(G,\hB)$ is called a Rota-Baxter Lie group of weight $-1$. 
\end{defi}

It was shown in \cite{LG} that the differentiation of a Rota-Baxter Lie group of weight $-1$ gives rise to a Rota-Baxter Lie algebra of weight $-1.$ 

Then, let us recall from \cite{LG} the definition of Rota-Baxter group homomorphisms. Let $(G,\hB)$ and $(G',\hB')$ be Rota-Baxter groups of weight $-1.$ A map $f:(G,\hB)\to (G',\hB')$ is called a Rota-Baxter group homomorphism if $f$ is a group homomorphism from $G$ to $G'$ satisfying
\begin{equation}\label{RBGH}
f\circ \hB=\hB'\circ f.
\end{equation}
Parallel to the case of Rota-Baxter Lie algebras, for any  Rota-Baxter operator of weight $-1$ on a group, we have the following proposition.
\begin{pro}\label{G-1}
Let $(G,B)$ be a Rota-Baxter group of weight $-1.$ Define the operator $\wtd{\hB}: G \to G$ by 
\begin{equation}\label{HBWB11}
\wtd{\hB}(a)=a\hB(a^{-1}),\ \forall a\in G.
\end{equation}
Then $\wtd{\hB}$ is a Rota-Baxter operator of weight $-1$ on $G.$
\end{pro}
\begin{proof}
For any $a,b\in G,$ we have 
\begin{equation}
\begin{aligned}
\wtd{\hB}(\Ad_{\wtd{\hB}(a)}(b)a)&=a\hB(a^{-1})b\hB(a^{-1})^{-1}a^{-1}a\hB\left(a^{-1}a\hB(a^{-1})b^{-1}\hB(a^{-1})^{-1}a^{-1}\right) \\
&=a\hB(a^{-1})b\hB(a^{-1})^{-1}\hB\left(\hB(a^{-1})b^{-1}\hB(a^{-1})^{-1}a^{-1}\right)\\
&=a\hB(a^{-1})b\hB(a^{-1})^{-1}\hB(a^{-1})\hB(b^{-1})\ (\textit{by \eqref{RBLS}})\\
&=a\hB(a^{-1})b\hB(b^{-1})=\wtd{\hB}(a)\wtd{\hB}(b).\\
\end{aligned}
\end{equation}
This proves $\wtd{\hB}$ is a Rota-Baxter operator of weight $-1$ on $G$ of weight $-1.$
\end{proof}
Conversely, let $(G,\hB)$ be a Rota-Baxter group of weight $-1.$ Then one can readily check that 
\begin{equation}\label{HBWB}
\hB(a)=a\wtd{\hB}(a^{-1}),\ \forall a\in G.
\end{equation}

Let $(G,\hB)$ be a Rota-Baxter group of weight $-1.$ Define $\cdot_{\hB}:G\times G\to G$ by 
$$a\cdot_{\hB} b=\Ad_{\hB(a)}(b)a,\ \forall a,b\in G.$$ By \cite{LG}, we know that $(G,\cdot_{\hB})$ is a group with identity $e$, it is called the descendent group of $(G,B).$ Denote this group by $G_{\hB}.$ It follows from \cite{LG} that $\hB$ and $\wtd{\hB}$ are both group homomorphisms between $G$ and $G_{\hB}$.

Next, let us recall the basic notions about matched pairs of groups given in \cite{MJ}.

\begin{defi}\label{MPG}
Let $G$ and $H$ be groups. Let $\rho:G\times H\to H$ be a group action of $G$ on $H$ and $\mu:H\times G\to G$ be a group action of $H$ on $G.$ Then the quadruple $(G,H,\rho,\mu)$ is called a matched pair of groups if  
the following conditions are satisfied:
\begin{itemize}
\item[(a)] $\rho(a^{-1})(b_1b_2)=(\rho(a^{-1})b_1)(\rho((\mu(b_1^{-1})a)^{-1})b_2),\ \forall a\in G,\ \forall b_1, b_2\in H;$
\item[(b)] $\mu(b^{-1})(a_1a_2)=(\mu(b^{-1})a_1)(\mu((\rho(a_1^{-1})b)^{-1})a_2),\ \forall a_1,a_2\in G, \ \forall b\in H;$
\item[(c)] $\rho(a)e'=e', \ \forall a\in G$;
\item[(d)] $\mu(b)e=e,\ \forall b\in H,$
\end{itemize}
where $e'$ denotes the identity of $H.$
\end{defi}

Recall from \cite{MJ} that the matched pair $(G,H,\rho,\mu)$ induces a Lie group structure on $G\times H$, with multiplication defined by
\begin{equation*}
\begin{aligned}
(a_1,b_1)(a_2,b_2)&=\left((\mu(b_2^{-1})a_1^{-1})^{-1}a_2,b_1(\rho(a_1)b_2)\right),\ \forall (a_1,b_1), (a_2,b_2)\in G\times H,\\
(a,b)^{-1}&=\left((\mu(b)a)^{-1},\rho(a^{-1})b^{-1}\right),\ \forall (a,b)\in G\times H.
\end{aligned}
\end{equation*}
This group is called the bicrossed product of groups $G$ and $H.$

Let $G$ be a group. For any $a\in G,$ define the adjoint action of $a$ on $G$ by 
\begin{equation*}
\Ad_a(b)=aba^{-1}, \ \forall b\in G.
\end{equation*}

Next, we recall the definition of homomorphisms of matched pairs of groups.

\begin{defi}
Let $(G_+,G_-,\rho,\mu)$ and $(G'_+, G'_-,\rhd',\mu')$ be matched pairs of groups. A homomorphism of matched pairs of groups between $(G_+,G_-,\rho,\mu)$ and $(G'_+, G'_-,\rhd',\mu')$ is a pair $(F_+, F_-)$ such that 
\begin{itemize}
\item[(a)] $F_+:G_+\to G'_+$ and $F_-: G_-\to G'_-$ are group homomorphisms;

\item[(b)] $F_+$ and $F_-$ satisfy
\begin{equation}\label{MPGH}
\begin{aligned}
F_+(\mu(b)(a))&=\mu'(F_-(b))(F_+(a)),\ \forall a\in G_+,\ \forall b\in G_-;\\
F_-(\rho(a)(b))&=\rho'(F_+(a))(F_-(b)),\ \forall a\in G_+,\ \forall b\in G_-.
\end{aligned}
\end{equation}
\end{itemize}
\end{defi}

\subsection{The matched pair of groups on a Rota-Baxter group}

Let $(G,B)$ be a Rota-Baxter Lie group of weight $-1.$ Denote $$
\begin{aligned}
G_+&=\im \hB, \quad G_-=\im \wtd{\hB}, \\
H_+&=\ker \wtd{\hB},\quad H_-=\ker \hB.
\end{aligned}
$$ By \cite{LG}, we know that $H_+\subseteq G_+$ and $H_-\subseteq G_-$ are both normal groups.

By \eqref{RBLS}, one can readily obtain the following lemma. It will be used in the proof of Theorem \ref{MGRB} and Lemma \ref{DAG}.
\begin{lem}\label{ID}
Let $(G,\hB)$ be a Rota-Baxter group of weight $-1.$ Then $\hB(e)=e.$
\end{lem}

Now, we state the group version of Theorem \ref{RLMP}.
\begin{thm}\label{MGRB}
Let $(G,B)$ be a Rota-Baxter group of weight $-1.$ Define the operator $\rho:G_+\times G_-\to G_-$ by 
\begin{equation}\label{MG1}
\rho(\hB(a))(\wtd{\hB}(b))=\wtd{\hB}(\Ad_{\hB(a)}(b)), \ \forall a,b\in G,
\end{equation}
and define the operator $\mu:G_-\times G_+\to G_+$ by 
\begin{equation}
\mu(\wtd{\hB}(a))(\hB(b))=\hB(\Ad_{\wtd{\hB}(a)}(b)),\ \forall a,b\in G.
\end{equation}
Then $(G_+,G_-,\rho,\mu)$ is a matched pair of groups.
\end{thm}
\begin{proof}
First, let us verify that $\rho$ is well defined. For any $a\in G$ and $b\in H_+$, we have 
\begin{equation}\label{MPG1}
\begin{aligned}
\rho(\hB(a))(\wtd{\hB}(b))&=\wtd{\hB}(\Ad_{\hB(a)}(b))=\wtd{\hB}(\hB(a)b\hB(a)^{-1})
\end{aligned}
\end{equation}
As $H_+$ is the normal subgroup of $G_+,$ we know that $\hB(a)b\hB(a)^{-1}\in H_+.$
Then we have $\rho(\hB(a))(\wtd{\hB}(b))=e.$ This shows that $\rho$ is well defined. It is similar to check that $\mu$ is also well defined. It follows that $\rho$ is the representation of $G_+$ on $G_-$ and $\mu$ is the representation of $G_-$ on $G_+.$ Next, we prove that (a) of Definition \ref{MPG} holds. For any $a,b\in G,$ we have
\begin{equation}\label{MGE1}
\rho(\hB(a)^{-1})(\wtd{\hB}(b))=\wtd{\hB}(\hB(a)^{-1}b\hB(a))=\hB(a)^{-1}b\hB(a)\hB(\hB(a)^{-1}b^{-1} \hB(a)).
\end{equation}
By \eqref{RBLS}, 
$$\hB(a)\hB(\hB(a)^{-1}b^{-1} \hB(a))=\hB(b^{-1}a).$$ It follows that $\hB(\hB(a)^{-1}b^{-1}\hB(a))=\hB(a)^{-1}\hB(b^{-1}a).$ Then by \eqref{MGE1}, we have 
\begin{equation}\label{MGE2}
\rho(\hB(a)^{-1})(\wtd{\hB}(b))=\hB(a)^{-1}b\hB(b^{-1}a). 
\end{equation}
And it is similar to prove that 
\begin{equation}\label{MGE3}
\mu(\wtd{\hB}(b)^{-1})(\hB(a))=\hB(\wtd{\hB}(b)^{-1}a\wtd{\hB}(b))=(\wtd{\hB}(a^{-1}b)^{-1}a^{-1}\wtd{\hB}(b))^{-1}. 
\end{equation}
It follows from \eqref{MGE2} and \eqref{MGE3} that 
\begin{equation}\label{MGE4}
\begin{aligned}
\wtd{\hB}(\hB(a)^{-1}b\hB(a))\hB(\wtd{\hB}(b)^{-1}a\wtd{\hB}(b))^{-1}&=\hB(a)^{-1}b\hB(b^{-1}a)\wtd{\hB}(a^{-1}b)^{-1}a^{-1}\wtd{\hB}(b)\\
&=\hB(a)^{-1}bb^{-1}aa^{-1}\wtd{\hB}(b)\ \textit{(by the definition of $\wtd{\hB}$)}\\
&=\hB(a)^{-1}\wtd{\hB}(b)
.
\end{aligned}
\end{equation}
Then for any $a,b_1,b_2\in G,$ we get 
\begin{equation*}
\begin{aligned}
&(\rho(\hB(a)^{-1})\wtd{\hB}(b_1))\left(\rho((\mu(\wtd{\hB}(b_1)^{-1})\hB(a))^{-1})\wtd{\hB}(b_2)\right)\\
=&\wtd{\hB}(\hB(a)^{-1}b_1\hB(a))\wtd{\hB}\left(\Ad_{\hB(\wtd{\hB}(b_1)^{-1}a\wtd{\hB}(b_1))^{-1}}b_2 \right)\\
=&\wtd{\hB}\left(\wtd{\hB}((\Ad_{\wtd{\hB}(\hB(a)^{-1}b_1\hB(a))\hB(\wtd{\hB}(b_1)^{-1}a\wtd{\hB}(b_1))^{-1}}b_2)\hB(a)^{-1}b_1\hB(a)\right)\ (\textit{by \eqref{RBLS}})\\
=&\wtd{\hB}\left(\wtd{\hB}((\Ad_{\hB(a)^{-1}\wtd{\hB}(b_1)}b_2)\hB(a)^{-1}b_1\hB(a)\right) \ (\textit{by \eqref{MGE4} }).\\
\end{aligned}
\end{equation*}
It follows that 
\begin{equation*}
\begin{aligned}
&(\rho(\hB(a)^{-1})\wtd{\hB}(b_1))\left(\rho((\mu(\wtd{\hB}(b_1)^{-1})\hB(a))^{-1})\wtd{\hB}(b_2)\right)\\
=&\wtd{\hB}\left(\hB(a)^{-1}\wtd{\hB}(b_1)b_2 \wtd{\hB}(b_1)^{-1}b_1\hB(a)\right)\\
=&\rho(\hB(a)^{-1}) \wtd{\hB}(\wtd{\hB}(b_1)b_2\wtd{\hB}(b_1)^{-1}b_1)\\
=&\rho(\hB(a)^{-1})(\wtd{\hB}(b_1)\wtd{\hB}(b_2)).
\end{aligned}
\end{equation*}
This proves (a) of Definition \ref{MPG} and it is similar to prove (b) of Definition \ref{MPG}. Finally, conditions (c) and (d) of Definition \ref{MPG} are readily verified. Therefore, $(G_+,G_-,\rho,\mu)$ is a matched pair of groups.
\end{proof}
The matched pair of groups $(G_+,G_-,\rho,\mu)$ given in the above theorem is called \textbf{the matched pair of groups on the Rota-Baxter group $(G,\hB).$}

By the definition of the matched pair of groups on a Rota-Baxter group of weight $-1$, we obtain the following Lemma. It will be used in the proof of Theorem \ref{TMG}

\begin{lem}\label{SLS}
Let $(G,\hB)$ be a Rota-Baxter group of weight $-1$ and $(G_+,G_-,\rho,\mu)$ be the matched pair of groups on $(G,\hB).$
Then for any $a,b\in G,$
\begin{equation*}
\left(\hB(a^{-1})^{-1},\wtd{\hB}(a)\right) \left(\hB(b^{-1})^{-1},\wtd{\hB}(b) \right)=\left(\hB(b^{-1}a^{-1}),\wtd{\hB}(ab) ) \right)
\end{equation*}
\end{lem}
\begin{proof}
By the definition, we have
\begin{equation*}
\begin{aligned}
&\left(\hB(a^{-1})^{-1},\wtd{\hB}(a)\right) \left(\hB(b^{-1})^{-1},\wtd{\hB}(b) \right)\\
=&\left( \left(\hB\left(\wtd{\hB}(b)^{-1}a^{-1}\wtd{\hB}(b) \right)      \right)^{-1}\hB(b^{-1})^{-1},  
\wtd{\hB}(a)\wtd{\hB}\left(
\hB(a^{-1})^{-1}b\hB(a^{-1})               \right)
\right)\\
=&\left( \left(\hB(b^{-1})\hB\left(\wtd{\hB}(b)^{-1}a^{-1}\wtd{\hB}(b) \right)\right)^{-1},  
\wtd{\hB}(a)\wtd{\hB}\left(
\hB(a^{-1})^{-1}b\hB(a^{-1})               \right)
\right)\\
=&\left( \left(\hB\left(\hB(b^{-1})\wtd{\hB}(b)^{-1}a^{-1}\wtd{\hB}(b)\hB(b^{-1})^{-1} b^{-1}\right) \right)^{-1},  
\wtd{\hB}\left(\wtd{\hB}(a)
\hB(a^{-1})^{-1}b\hB(a^{-1})\wtd{\hB}(a)^{-1}a\right)\right)\ (\textit{by \eqref{RBLS}})\\
\end{aligned}
\end{equation*}
for any $a,b\in G.$
It follows that 
\begin{equation*}
\begin{aligned}
&\left(\hB(a^{-1})^{-1},\wtd{\hB}(a)\right) \left(\hB(b^{-1})^{-1},\wtd{\hB}(b) \right)\\
=&\left( \left(\hB\left(b^{-1}a^{-1}bb^{-1}\right) \right)^{-1},  
\wtd{\hB}\left(aba^{-1}a\right)\right)\ (\textit{by \eqref{RBLS}})
\\
=&\left( \left(\hB\left(b^{-1}a^{-1}\right) \right)^{-1},  
\wtd{\hB}\left(ab\right)\right)\ (\textit{by \eqref{HBWB11} and \eqref{HBWB}}).
\\
\end{aligned}
\end{equation*}
\end{proof}

In the next proposition, we show that every homomorphism of Rota-Baxter groups of weight $-1$ induces a homomorphism of matched pairs of groups.

\begin{pro}
Let $(G,\hB)$, $(G',\hB')$ be Rota-Baxter groups of weight $-1$. Let $(G_+,G_-,\rhd,\bhd)$ and $(G'_+,G'_-,\rhd',\bhd')$ be the matched pairs of groups on $(G,\hB)$ and $(G',\hB')$ respectively. Let $F:(G,\hB)\to (G', \hB')$ be a homomorphism of Rota-Baxter groups. Let $F_+:G_+\to G'_+$ and $F_-:G_-\to G'_-$ be the restrction of $F$ to $G_+$ and $G_-$ respectively.
Then $(F_+,F_-)$ is a homomorphism of matched pairs of groups.
\end{pro}
\begin{proof}
It is easy to see that $F_+$ and $F_-$ are both group homomorphisms. Then we prove that \eqref{MPGH} holds. For any $a,b\in G,$ we have
$$
\begin{aligned}
\mu'(F_+(\hB(a))) (F_-(\wtd{\hB}(b)))
&=\mu'(\hB'(F(a)))(\wtd{\hB'}(F(b)))\ (\textit{by \eqref{RBGH}})\\
&=\wtd{\hB'}(\Ad_{\hB'(F(a))} F(b) )\\
&=\wtd{\hB'}(\Ad_{F(\hB(a))} F(b) )\ (\textit{by \eqref{RBGH}})\\
&=\wtd{\hB'}\circ F_-(\Ad_{\hB(a)} b).
\end{aligned}
$$
Then we have 
$$
\begin{aligned}
\mu'(F_+(\hB(a))) (F_-(\wtd{\hB}(b)))
&=F\circ \wtd{\hB}(\Ad_{\hB(a)} b)\ (\textit{by \eqref{RBGH}})\\
&=F\left(\mu(\hB(a))(\wtd{\hB}(b))\right).
\end{aligned}
$$
It is similar to prove that
$$\mu'(F_-(\wtd{\hB}(a)))(F_+(\hB(b)))=F_+(\mu(\wtd{\hB}(a))(\hB(b))).$$ Therefore, $(F_+,F_-)$ is a homomorphism of matched pairs of groups.
\end{proof}

Let $(G,\hB)$ be a Rota-Baxter Lie group of weight $-1$ and $(\frkg)$ be the infinitesimal Lie algebra of $G.$  Recall from \cite{LG} that the tangent map of $\hB$ is given by 
\begin{equation*}
B(x)=\hB_{\ast e}(\exp(tx))=\left.\frac{d}{d t}\right|_{t=0}\hB( \exp(tx)), \ \forall x\in \frkg.
\end{equation*}

The matched pair of groups $(G_+,G_-,\rho,\mu)$ given in the above theorem is called the matched pair of groups on $(G_+,G_-).$

Let $(G_+,G_-,\rho,\mu)$ be the matched pair of groups on $(G,\hB).$ It follows directly from the definition that $\rho$ and $\mu$ are smooth maps. In the next proposition, we establish the relationship between the matched pair of groups on a Rota-Baxter group and the matched pair of Lie algebras on a Rota-Baxter Lie algebra.

\begin{pro}
Let $(G,\hB)$ be a Rota-Baxter Lie group of weight $-1$ whose infinitesimal Rota-Baxter Lie algebra is $(\frkg,B=\hB_{\ast e})$ and $(G_+,G_-,\rho,\mu)$ be the matched pair of groups on $(G,\hB).$ Let $\frkg_+=\im B$,  $\frkg_-=\im \wtd{B}$. Let $\rhd:\frkg_+\times \frkg_-\to \frkg_-$ be the operator given by
\begin{equation*}
B(x)\rhd \wtd{B}(y)=\left.\frac{d}{d t}\right|_{t=0}\left.\frac{d}{d s}\right|_{s=0}\rho(\hB(\exp(tx)))(\wtd{\hB}(\exp(sy)
)),\ \forall x,y\in \frkg,
\end{equation*}
and $\bhd:\frkg_-\times\frkg_+\to \frkg_+$ be the operator given by 
\begin{equation*}
\wtd{B}(x)\bhd B(y)=\left.\frac{d}{d t}\right|_{t=0}\left.\frac{d}{d s}\right|_{s=0}\mu(\wtd{\hB}(\exp(tx)))(\hB(\exp(sy)
)),\ \forall x,y\in \frkg.
\end{equation*}

Then $(\frkg_+,\frkg_-,\rhd,\bhd)$ is the matched pair of Lie algebras on $(\frkg,B).$ 

\end{pro}
\begin{proof}
It follows from the definition of $B$ that for any $x,y\in\frkg,$

\begin{equation*}
\begin{aligned}
B(x)\rhd \wtd{B}(y)&=\left.\frac{d}{d t}\right|_{t=0}\left.\frac{d}{d s}\right|_{s=0}\rho(\hB(\exp(tx)))(\wtd{\hB}(\exp(sy)
))\\
&=\left.\frac{d}{d t}\right|_{t=0}\left.\frac{d}{d s}\right|_{s=0}\wtd{\hB}\left(\Ad_{\hB(\exp(tx))}\exp(sy)\right)\\
&=\wtd{\hB}_{\ast e} \left( \left.\frac{d}{d t}\right|_{t=0}\left.\frac{d}{d s}\right|_{s=0} \Ad_{\hB(\exp(tx))} \exp(sy)         \right)\\
&=\wtd{B}([B(x),y]).
\end{aligned}
\end{equation*}
It is similar to prove that $\wtd{B}(x)\bhd B(y)=B([\wtd{B}(x),y]).$ Then we know that $(\frkg_+,\frkg_-,\rhd,\bhd)$ is the matched pair of Lie algebras on $(\frkg,B).$

\end{proof}

\subsection{Rota-Baxter groups of weight $-1$ and group projections on matched pairs of groups}
Let $(G,\hB)$ be a Rota-Baxter group of weight $-1.$ In this section, we denote $a^{\dagger}=\Ad_{\hB(a)} a^{-1}$ for any $a\in G.$
\begin{defi}\label{GPMD}
Let $(G_+,G_-,\rho,\mu)$ be a matched pair of groups. An operator $\mathcal{C}:G_+\bowtie G_-\to G_+\bowtie G_-$ is called a \textbf{group projection on $(G_+,G_-,\rho,\mu)$} if $\hC:G_+\bowtie G_-\to G_+\bowtie G_-$ is an idempotent group homomorphism and the operator $\wtd{\hC}:G_+\bowtie G_-\to G_+\bowtie G_-$ given by 
$$\wtd{\hC}((a,b))=(a,b)\hC((a,b)^{-1}), \ \forall (a,b)\in G_+\bowtie G_-$$
is also an idempotent group homomorphism.
\end{defi}

It is not hard to check that $\hC$ is a Rota-Baxter operator of weight $-1$ on $G_+\bowtie G_-.$ Then it follows from Proposition \ref{G-1} that $\wtd{\hC}$ is a Rota-Baxter operator of weight $-1$ on $G_+\bowtie G_-.$

By the definition of group projections on matched pairs of groups, we have the following proposition.
\begin{pro}\label{GPMG}
Let $(G_+,G_-,\rho,\mu)$ be a matched pair of groups and $\mathcal{C}:G_+\bowtie G_-\to G_+\bowtie G_-$ be a group projection on $(G_+,G_-,\rho,\mu)$. Define $\wtd{C}:G_+\bowtie G_-\to G_+\bowtie G_-$ by
$$\wtd{\hC}((a,b))=(a,b)\hC((a,b)^{-1}), \ \forall (a,b)\in G_+\bowtie G_-.$$ Then the following holds:
\begin{itemize}
\item[(a)] $\hC$ and $\wtd{\hC}$ are Rota-Baxter operators of weight $-1$ on $G_+\bowtie G_-$ such that
\begin{equation}\label{IDC}
\wtd{\hC}((a,b))\hC((a,b))=\hC((a,b))\wtd{\hC}((a,b))=(a,b)
\end{equation}
for any $(a,b)\in G_+\bowtie G_-;$
\item[(b)]
$\wtd{\hC}$ is a group projection on $(G_+,G_-,\rhd,\bhd)$.
\end{itemize}
\end{pro}
\begin{proof}
Denote the identity elements of $G_+$ and $G_-$ by $e$ and $e'$, respectively. 
For any $a,b\in G_+\bowtie G_-$, we have
\begin{equation*}
\hC(\hC(a)b\hC(a)^{-1}a)=\hC(\hC(a))\hC(b)\hC(\hC(a)^{-1})\hC(a).
\end{equation*}
As $\hC$ is an idempotent group homomorphism, it follows that 
\begin{equation*}
\hC(\hC(a)b\hC(a)^{-1}a)=\hC(a)\hC(b)=\hC(ab).
\end{equation*}
This shows that $\hC$ and $\wtd{\hC}$ are both Rota-Baxter operators of weight $-1$. By the definition, we know that $\wtd{\hC}(a)\hC(a)=a$ for any $a\in G_+\bowtie G_-$. Then by 
\eqref{HBWB11} and \eqref{HBWB}, we have $\hC(a)\wtd{\hC}(a)=a.$
Finally, it follows from the definition that $\wtd{\hC}$ is a group projection on $(G_+,G_-,\rhd,\bhd).$
\end{proof}

Then we give an equivalent characterisation of group projections on matched pairs of groups.
\begin{pro}\label{GMEQ}
Let $(G_+,G_-,\rho,\mu)$ be a matched pair of groups and $\mathcal{C}:G_+\bowtie G_-\to G_+\bowtie G_-$ be an idempotent group homomorphism and $\wtd{\hC}:G_+\bowtie G_-\to G_+\bowtie G_-$ be the operator given by
$$\wtd{\hC}((a,b))=(a,b)\hC((a,b)^{-1}), \ \forall (a,b)\in G_+\bowtie G_-.$$ Then the following statements are equivalent:
\begin{itemize}
\item[(a)] $\hC$ is a group projection on $(G_+,G_-,\rho,\mu);$
\item[(b)] For any $(a_1,b_1),(a_2,b_2)\in G_+\bowtie G_-,$
\begin{equation}\label{EQE1}
\hC((a_1,b_1))\wtd{\hC}((a_2,b_2))=\wtd{\hC}((a_2,b_2))\hC((a_1,b_1)).
\end{equation}
\end{itemize}
\end{pro}
\begin{proof}
(a)$\Rightarrow$ (b) The proof is obvious.

(b)$\Rightarrow$(a) 
First, let us prove that $\wtd{C}$ is idempotent. By the defintion, we have $$\hC(\wtd{\hC}((a,b)))=\hC((a,b)\hC((a,b)^{-1}))=\hC((a,b))\hC(\hC((a,b)^{-1}))=\hC((a,b))\hC((a,b))^{-1}=e$$ for any $(a,b)\in G_+\bowtie G_-.$ It follows that
$$\wtd{\hC}^{2}((a,b))=\wtd{\hC}((a,b))\hC(\wtd{\hC}((a,b))^{-1})=\wtd{\hC}((a,b))\hC(\wtd{\hC}((a,b)))^{-1}=\wtd{\hC}((a,b)).$$
This shows that $\wtd{\hC}$ is idempotent. Then we prove that $\wtd{\hC}$ is a group homomorphism. It follows from the definition and \eqref{EQE1} that  
$$
\begin{aligned}
(a_1,b_1)(a_2,b_2)&=\wtd{\hC}((a_1,b_1))\hC((a_1,b_1))\wtd{\hC}((a_2,b_2))\hC((a_2,b_2))\\&=\wtd{\hC}((a_1,b_1))\wtd{\hC}((a_2,b_2))\hC((a_1,b_1))\hC((a_2,b_2)).
\end{aligned}
$$
On the other hand, we have
$$
\begin{aligned}
(a_1,b_1)(a_2,b_2)&=\wtd{\hC}((a_1,b_1)(a_2,b_2))\hC((a_1,b_1)(a_2,b_2))\\&=\wtd{\hC}((a_1,b_1)(a_2,b_2))\hC((a_1,b_1))\hC((a_2,b_2))
\end{aligned}
$$
This implies $\wtd{\hC}((a_1,b_1)(a_2,b_2))=\wtd{\hC}((a_1,b_1))\wtd{\hC}((a_2,b_2)).$ It follows that $\wtd{\hC}$ is a group homomorphism.
Therefore, $\hC$ is a group projection on $(G_+,G_-,\rho,\mu).$
\end{proof}

Now, we show that if there is a group projection on a matched pair of groups $(G_+,G_-,\rhd,\bhd)$, then there is a direct sum decomposition on $G_+\bowtie G_-$. 

\begin{pro}\label{DSGP}
Let $(G_+,G_-,\rhd,\bhd)$ be a matched pair of groups. Let $\hC$ be a group projection on $(G_+,G_-,\rhd,\bhd)$. Denote $G_1=\im \hC$ and $G_2=\im \wtd{\hC}$. Then $G_+\bowtie G_- =  G_1\oplus G_2$, where $G_1\oplus G_2$ is the direct sum of $G_1$ and $G_2$ as groups.
\end{pro}
\begin{proof}
The proof follows directly from Definition \ref{GPMD} and Proposition \ref{GPMG}.

\end{proof}

The next Proposition is the group analogue of Proposition \ref{MR}.
\begin{pro}\label{GS}
Let $(G_+,G_-,\rho,\mu)$ be a matched pair of groups and $\mathcal{C}$ be a group projection on $(G_+,G_-,\rho,\mu)$. Let $G=\im \mathcal{C}$. Denote the identity element of $G_+$ by $e$. Define the operator
$\hB_1:G\to G$ by
\begin{equation*}
\hB((a,b))=\hC((e,b)),\ \forall (a,b)\in G,
\end{equation*}
and the operator $\hB_2:G\to G$ by
\begin{equation*}
\wtd{\hB}((a,b))=(a,b)\hC((e,b^{-1})),\ \forall (a,b)\in G.
\end{equation*}
Then $\hB_1$ and $\hB_2$ are both Rota-Baxter operators of weight $-1$ on $G$ such that 
$$\hB((a,b))\wtd{\hB}((a,b)^{-1})^{-1}=(a,b)$$ for any $(a,b)\in G$.
\end{pro}
\begin{proof}
Denote the identity element of $G_-$ by $e'$. For any $(a_1,b_1),(a_2,b_2)\in G,$ we have
\begin{equation*}
\begin{aligned}
&\hB\left(\hB((a_1,b_1))(a_2,b_2)\hB((a_1,b_1))^{-1})(a_1,b_1)  \right)\\
=&\hB\left( \hC((e,b_1))(a_2,b_2)\hC((e,b_1))^{-1}(a_1,b_1)\right)\\
=&\hB\left( \hC((e,b_1))\hC((a_2,b_2)\hC((e,b_1)^{-1}) )\hC((a_1,b_1))\right)\ (\textit{as $\hC$ is idempotent})\\
=&\hB\circ \hC\left((e,b_1)(a_2,b_2)(e,b_1)^{-1} (a_1,b_1)\right)\\
=&\hB\circ \hC\left((e,b_1)(a_2,b_2)(e,b_1)^{-1} (e,b_1)(a_1,e')\right)\\
=&\hB\circ \hC\left((e,b_1)(a_2,b_2)(a_1,e')\right).
\end{aligned}
\end{equation*}
By the proof of Proposition \ref{GPMG}, we have
$$
\begin{aligned}
(e,b_1)(a_2,b_2) (a_1,e')
=&\hC\left((e,b_1)(a_2,b_2) (a_1,e')\right)\wtd{\hC}\left((e,b_1)(a_2,b_2) (a_1,e')\right)\\
=&\hC\left((e,b_1)(a_2,b_2) (a_1,e')\right)\wtd{\hC}\left((e,b_1)\right)\wtd{\hC}\left((a_2,b_2)\right)\wtd{\hC}\left((a_1,e')\right)\\
=&\hC\left((e,b_1)(a_2,b_2) (a_1,e')\right)\wtd{\hC}\left((e,b_1)\right)(e,e')\wtd{\hC}\left((a_1,e')\right)\\
=&\hC\left((e,b_1)(a_2,b_2) (a_1,e')\right)\wtd{\hC}(a_1,b_1)\\
=&\hC\left((e,b_1)(a_2,b_2) (a_1,e')\right).
\end{aligned}
$$
This shows that $(e,b_1)(a_2,b_2) (a_1,e')\in G.$
It follows that
\begin{equation*}
\begin{aligned}
&\hB\left(\hB((a_1,b_1))(a_2,b_2)\hB((a_1,b_1))^{-1})(a_1,b_1)  \right)\\
=&\hB\left((e,b_1)(a_2,b_2) (a_1,e')\right)\\
=&\hB\left((e,b_1)(e,b_2)(a_2,e')(a_1,e')\right)\\
=&\hB\left((e,b_1b_2)(a_2a_1,e')\right).
\end{aligned}
\end{equation*}
And then we have
\begin{equation*}
\begin{aligned}
&\hB\left(\hB((a_1,b_1))(a_2,b_2)\hB((a_1,b_1))^{-1})(a_1,b_1)  \right)\\
=&\hB((a_2a_1,b_1b_2))
=\hC\left((e,b_1b_2)\right)\\
=&\hC((e,b_1))\hC((e,b_2))=\hB((a_1,b_2))\hB((a_2,b_2)).
\end{aligned}
\end{equation*}
This proves that $\hB$ is a Rota-Baxter operator of weight $-1$ on $G$. Finally, by Proposition \ref{G-1}, we know that $\wtd{\hB}$ is a Rota-Baxter operator of weight $-1$ on $G$. 
\end{proof}

The next two lemmas will be used in the proof of Theorem \ref{TMG}.
\begin{lem}\label{DAG}
Let $(G,\hB)$ be a Rota-Baxter group of weight $-1.$ Then for any $a\in G$, $\hB(a^{\dagger})=\hB(a)^{-1}.$ 
\end{lem}
\begin{proof}
For any $a\in G,$ by \eqref{RBLS},   
\begin{equation*}
\begin{aligned}
\hB(a)\hB(a^{\dagger})=\hB(\hB(a)\hB(a)^{-1}a^{-1}\hB(a)\hB(a)^{-1}a)=\hB(e).
\end{aligned}
\end{equation*}
It follows from Lemma \ref{ID} that $\hB(e)=e$. Therefore $\hB(a)\hB(a^{\dagger})=e$ and it follows that $\hB(a^{\dagger})=\hB(a)^{-1}.$
\end{proof}

By the definition of $\wtd{\hB}$, it is straightforward to derive the following lemma. It will be used in the proof of Proposition \ref{C3} and Theorem \ref{TMG}.
\begin{lem}\label{IV}
Let $(G,\hB)$ be a Rota-Baxter group of weight $-1.$ Then 
for any $a\in G,$ $$\hB(a)\wtd{\hB}(a^{-1})^{-1}=a. $$
\end{lem}

The next three lemmas will be used in the proof of Proposition \ref{C3} and Theorem \ref{TMG}.

\begin{lem}\label{DEC}
Let $(G,\hB)$ be a Rota-Baxter group of weight $-1.$ Then for any $a,b\in G,$
$a\cdot_{\hB}b=(a^{-1}\cdot_{\hB} b^{-1})^{-1}$
\end{lem}
\begin{proof}
For any $a,b\in G,$ we have 
\begin{equation*}
\begin{aligned}
a\cdot_{\wtd{\hB}}b&=a\hB(a^{-1})b\hB(a^{-1})^{-1}a^{-1}a\\
&=a\hB(a^{-1})b\hB(a^{-1})^{-1}\\
&=(\hB(a^{-1})b^{-1}\hB(a^{-1})^{-1}a^{-1})^{-1}\\
&=(a^{-1}\cdot_{\wtd{\hB}}b)^{-1}.
\end{aligned}
\end{equation*}
\end{proof}
\begin{lem}\label{IS}
Let $(G,\hB)$ be a Rota-Baxter group of weight $-1.$ Then the following identity holds:
$$\wtd{\hB}(\hB(a))=\hB(\wtd{\hB}(a^{-1})^{-1}), \ \forall a\in G.$$
\end{lem}
\begin{proof}
The proof follows from the definition.
\end{proof}

\begin{lem}\label{TE2}
Let $(G,\hB)$ be a Rota-Baxter group of weight $-1$ and $(G_+,G_-,\rho,\mu)$ be the matched pair of groups on $(G,\hB).$ Then the following identity holds:
\begin{equation*}
\left(\hB\left( \wtd{\hB}(a)\right)^{-1},
\wtd{\hB}\left(\hB(a^{-1})^{-1}   \right)
\right)\left(\hB\left( \wtd{\hB}(b)\right)^{-1},
\wtd{\hB}\left(\hB(b^{-1})^{-1}\right)
\right)=\left( \hB(\wtd{\hB}(a)),\wtd{\hB}(\hB(b^{-1})^{-1})  \right)
\end{equation*}
for any $a,b\in G.$ 
\end{lem}

\begin{proof}
For any $a,b\in G,$ we have
\begin{equation*}
\begin{aligned}
&\left(\hB\left( \wtd{\hB}(a)\right)^{-1},
\wtd{\hB}\left(\hB(a^{-1})^{-1}   \right)
\right)\left(\hB\left( \wtd{\hB}(b)\right)^{-1},
\wtd{\hB}\left(\hB(b^{-1})^{-1}\right)
\right)\\
=&\left(\left( \hB\left(\wtd{\hB}\left(\hB(b^{-1})\right)^{-1} \wtd{\hB}(a)\wtd{\hB}\left(\hB(b^{-1})\right)\right)\right)^{-1}   \hB\left(\wtd{\hB}(b)
\right)^{-1},\wtd{\hB}\left(\hB(a^{-1})^{-1}\right)\wtd{\hB}\left(\hB(\wtd{\hB}(a))^{-1}\hB(b^{-1})^{-1} \hB(\wtd{\hB}(a))\right)     
\right)\\
=&\left(\hB\left(\wtd{\hB}(b)
\right)\left( \hB\left(\wtd{\hB}\left(\hB(b^{-1})^{-1}\right)^{-1} \wtd{\hB}(a)\wtd{\hB}\left(\hB(b^{-1})^{-1}\right)\right)   \right)^{-1},\wtd{\hB}\left(\hB(a^{-1})^{-1}\right)\wtd{\hB}\left(\hB(\wtd{\hB}(a))^{-1}\hB(b^{-1})^{-1} \hB(\wtd{\hB}(a))\right)     
\right)\\
=&\left( \hB(\wtd{\hB}(a)\wtd{\hB}(b))^{-1},\wtd{\hB}(\hB(b^{-1})^{-1}\hB(a^{-1})^{-1})  \right)\ (\textit{by \eqref{RBLS} and Lemma \ref{IV}}).
\end{aligned}
\end{equation*}
\end{proof}

The next proposition will be used in the proof of Theorem \ref{TMG}.

\begin{pro}\label{C3}
Let $(G,\hB)$ be a Rota-Baxter group of weight $-1$, and let $(G_+,G_-,\rho,\mu)$ be the matched pair of groups on $(G,\hB).$ Then the following identity holds:
$$\left(\hB(a)^{-1},\wtd{\hB}(a^{-1})\right) \left(\hB(\wtd{\hB}(b))^{-1}, \wtd{\hB}(\hB(b^{-1})^{-1})\right)= \left(\hB(\wtd{\hB}(b))^{-1}, \wtd{\hB}(\hB(b^{-1})^{-1})\right)\left(\hB(a)^{-1},\wtd{\hB}(a^{-1})\right),$$
where $a,b\in G.$
\end{pro}

\begin{proof}
For any $a,b\in G,$ we have 
\begin{equation*}
\begin{aligned}
&\left(\hB(a)^{-1},\wtd{\hB}(a^{-1})\right) \left(\hB(\wtd{\hB}(b))^{-1}, \wtd{\hB}(\hB(b^{-1})^{-1})\right)\\
=&\left(\hB\left(\Ad_{\wtd{\hB}(\hB(b^{-1})^{-1})^{-1}} a\right)^{-1}\hB\left(\wtd{\hB}(b)\right)^{-1}, \wtd{\hB}(a^{-1})\wtd{\hB}\left(\Ad_{\hB(a)^{-1}} \hB(b^{-1})^{-1}\right)  \right)\\
=&\left((\hB(\wtd{\hB}(b))\hB\left(\Ad_{\wtd{\hB}(\hB(b^{-1})^{-1})^{-1}} a)\right)^{-1}, \wtd{\hB}(a^{-1})\wtd{\hB}\left(\Ad_{\hB(a)^{-1}} \hB(b^{-1})^{-1}\right)  \right)\\
=&\left(\hB\left((\Ad_{\wtd{\hB}(\hB(b^{-1})^{-1})\wtd{\hB}(\hB(b^{-1})^{-1})^{-1}} a)\wtd{\hB}(b)\right)^{-1}, \wtd{\hB}\left((\Ad_{\wtd{\hB}(a^{-1})\hB(a)^{-1}} \hB(b^{-1})^{-1})a^{-1}\right)  \right)\ (\textit{by \eqref{RBLS}})\\
=&\left(\hB\left( a\wtd{\hB}(b))\right)^{-1}, \wtd{\hB}\left(a^{-1} \hB(b^{-1})^{-1}\right)  \right)\ (\textit{by \eqref{RBLS} and Lemma \ref{IV}}).\\
\end{aligned}
\end{equation*}
On the other hand, we have 
\begin{equation*}
\begin{aligned}
&\left(\hB(\wtd{\hB}(b))^{-1}, \wtd{\hB}(\hB(b^{-1})^{-1}\right)  \left(\hB(a)^{-1},\wtd{\hB}(a^{-1})\right) \\
=&\left( \hB\left(\Ad_{\wtd{\hB}(a^{-1})^{-1}} \wtd{\hB}(b)\right)^{-1}\hB(a)^{-1},   \wtd{\hB}(\hB(b^{-1})^{-1})\wtd{\hB}\left(\Ad_{\hB(\wtd{\hB}(b))^{-1}}a^{-1}\right) \right)\\
=&\left( (\hB(a)\hB\left(\Ad_{\wtd{\hB}(a^{-1})^{-1}} \wtd{\hB}(b))\right)^{-1},   \wtd{\hB}(\hB(b^{-1})^{-1})\wtd{\hB}\left(\Ad_{\hB(\wtd{\hB}(b))^{-1}}a^{-1}\right) \right) \ (\textit{by \eqref{RBLS}})   \\
=&\left( \hB\left(a\wtd{\hB}(b)\right)^{-1},   \wtd{\hB}\left(a^{-1} \hB(b^{-1})^{-1}\right) \right)\ (\textit{by \eqref{RBLS} and  Lemma \ref{IV}}).
\end{aligned}
\end{equation*}

\end{proof}

Now, we give the main result of this section. The next theorem is the group version of \ref{FL}.
\begin{thm}\label{TMG}
Let $(G,\hB)$ be a Rota-Baxter group of weight $-1$ and $(G_+,G_-,\rho,\mu)$ be the matched pair of groups on $(G,\hB).$ Then the operator $\hC:G_+\bowtie G_-\to G_+\bowtie G_-$ defined by 
\begin{equation*}
\hC\left(\left(\hB(a),\wtd{\hB}(b)\right)\right)=\left(\hB\left(\hB(a)^{-1}\wtd{\hB}(b)^{-1}\right)^{-1},\wtd{\hB}\left(\wtd{\hB}(b)\hB(a)\right)\right),\ \forall a,b\in G,
\end{equation*}
and the operator $\wtd{\hC}:G_+\bowtie G_-\to G_+\bowtie G_-$ defined by
\begin{equation*}
\wtd{\hC}\left(\left(\hB(a),\wtd{\hB}(b)\right)\right)=\left(\hB\left(\wtd{\hB}(b)\wtd{\hB}(a^{-1})\right)^{-1},\wtd{\hB}\left(\hB(a)^{-1}\hB(b^{-1})^{-1}\right)\right),\ \forall a,b\in G,
\end{equation*}
are group projections on $(G_+,G_-,\rho,\mu)$ such that $$\left(\hC\left(\hB(a),\wtd{\hB}(b)\right)\cdot_{\bowtie} \wtd{\hC}\left((\hB(a),\wtd{\hB}(b))\right) \right)=(\hB(a),\wtd{\hB}(b)).$$  
\end{thm}

\begin{proof}
First, by definition, $\hC$ and $\wtd{\hC}$ are well defined. Using Lemma \ref{IV}, it is straightforward to check that they are idempotent. Then we prove that $\hC$ is a group homomorphism.

For any $a_1,a_2,b_1,b_2\in G,$ we have
\begin{equation*}
\begin{aligned}
&\hC\left((\hB(a_1),\wtd{\hB}(b_1))\right) \hC\left((\hB(a_2),\wtd{\hB}(b_2))\right)\\
=&\left(\hB\left(\hB(a_1)^{-1}\wtd{\hB}(b_1)^{-1}\right)^{-1},\wtd{\hB}\left(\wtd{\hB}(b_1)\hB(a_1)\right)\right) \left(\hB\left(\hB(a_2)^{-1}\wtd{\hB}(b_2)^{-1}\right)^{-1},\wtd{\hB}\left(\wtd{\hB}(b_2)\hB(a_2)\right)\right)\\
=&\Big{(} (\hB\left(\Ad_{\wtd{\hB}(\wtd{\hB}(b_2)\hB(a_2))^{-1}} \hB(a_1)^{-1}\wtd{\hB}(b_1)^{-1}\right)^{-1}\hB\left(\hB(a_2)^{-1}\wtd{\hB}(b_2)^{-1}\right)^{-1}, \wtd{\hB}\left(\wtd{\hB}(b_1) \hB(a_1)\right)\\ &\quad \wtd{\hB}\left(\Ad_{\hB(\hB(a_1)^{-1}\wtd{\hB}(b_1)^{-1})^{-1}}\wtd{\hB}(b_2)\hB(a_2)\right)\Big{)}\\
=&\Big{(}\left(\hB(\hB(a_2)^{-1}\wtd{\hB}(b_2)^{-1}\right)\hB\left(\Ad_{\wtd{\hB}(\wtd{\hB}(b_2)\hB(a_2))^{-1}} \hB(a_1)^{-1}\wtd{\hB}(b_1)^{-1}\right)^{-1}, \wtd{\hB}\left(\wtd{\hB}(b_1) \hB(a_1)\right)\\ &\quad\wtd{\hB}\left(\Ad_{\hB(\hB(a_1)^{-1}\wtd{\hB}(b_1)^{-1})^{-1}}\wtd{\hB}(b_2)\hB(a_2)\right) \Big{)}.\\
\end{aligned}
\end{equation*}
It follows that
\begin{equation*}
\begin{aligned}
&\hC\left((\hB(a_1),\wtd{\hB}(b_1))\right) \hC\left((\hB(a_2),\wtd{\hB}(b_2))\right)\\
=&\left(\hB\left((\Ad_{\hB(a_2))^{-1}\wtd{\hB}(b_2)^{-1}} \hB(a_1)^{-1}\wtd{\hB}(b_1)^{-1})\hB(a_2)^{-1}\wtd{\hB}(b_2)^{-1})^{-1}\right),\wtd{\hB}\left((\Ad_{\wtd{\hB}(b_1) \hB(a_1)}\wtd{\hB}(b_2)\hB(a_2))\wtd{\hB}(b_1) \hB(a_1) \right)\right)\\&(\textit{by \eqref{RBLS} and Lemma \ref{IV}})
\\=&\left( \hB\left(\hB(a_2)^{-1}\wtd{\hB}(b_2)^{-1}\hB(a_1)^{-1}\wtd{\hB}(b_1)^{-1}\right),\wtd{\hB}\left(\wtd{\hB}(b_1)\hB(a_1)\wtd{\hB}(b_2)\hB(a_2)\right)\right)
\end{aligned}
\end{equation*}

On the other hand, we have 
\begin{equation*}
\begin{aligned}
&\hC\left(\left(\hB(a_1),\wtd{\hB}(b_1)\right) \left(\hB(a_2),\wtd{\hB}(b_2)\right)\right)\\
=&\hC\left(\hB(\Ad_{\wtd{\hB}(b_2)^{-1}}a_1^{\dagger})^{-1}\hB(a_2), \wtd{\hB}(b_1)\wtd{\hB}(\Ad_{\hB(a_1)} b_2) \right)\ (\textit{by Lemma \ref{DAG}})\\
=&\Big{(}\left(\hB\left(\wtd{\hB}(b_1)\wtd{\hB}(\Ad_{\hB(a_1)} b_2)\hB(\Ad_{\wtd{\hB}(b_2)^{-1}}a_1^{\dagger})^{-1}\hB(a_2)\right)^{-1}\right)^{-1}, \wtd{\hB}\left(\wtd{\hB}(b_1)\wtd{\hB}(\Ad_{\hB(a_1)} b_2)\hB(\Ad_{\wtd{\hB}(b_2)^{-1}}a_1^{\dagger})^{-1}\hB(a_2)\right)\Big{)}\\
=&\left( \hB\left((\wtd{\hB}(b_1)\hB(a_1)\wtd{\hB}(b_2)\hB(a_2))^{-1}\right)^{-1},\wtd{\hB}\left(\hB(\wtd{\hB}(b_1)\hB(a_1)\wtd{\hB}(b_2)\hB(a_2) )\right)\right)\ (\textit{by \eqref{MGE4}}).
\end{aligned}
\end{equation*}
It follows that $$\hC\left((\hB(a_1),\wtd{\hB}(b_1))(\hB(a_2),\wtd{\hB}(b_2))\right)=\hC\left((\hB(a_1),\wtd{\hB}(b_1))\right)\hC\left((\hB(a_2),\wtd{\hB}(b_2))\right).$$

Next, we prove that $\wtd{\hC}$ is an idempotent group homomorphism. For any $a,b\in G,$ we have
\begin{equation}\label{C1}
\begin{aligned}
\wtd{\hC}\left((\hB(a),\wtd{\hB}(b))\right)&=\left(\hB\left(\wtd{\hB}(b)\wtd{\hB}(a^{-1})\right)^{-1},\wtd{\hB}\left(\hB(a)^{-1}\hB(b^{-1})^{-1}\right)\right)\\
&=\left(\hB\left(\wtd{\hB}(b\cdot_{\wtd{\hB}} a^{-1})\right)^{-1},\wtd{\hB}\left(\hB(b^{-1}\cdot_{\hB} a)^{-1}\right)  \right)\ (\textit{by \eqref{RBLS}})\\
&=\left(\hB\left(\wtd{\hB}(b\cdot_{\wtd{\hB}} a^{-1})\right)^{-1},\wtd{\hB}\left(\hB((b\cdot_{\wtd{\hB}} a^{-1})^{-1})^{-1}\right)  \right)\ (\textit{by Lemma \ref{IS} }).
\end{aligned}
\end{equation}

Next, we have
\begin{equation}\label{C2}
\begin{aligned}
\hC\left((\hB(a),\wtd{\hB}(b))\right)=&\left(\hB(\hB(a)^{-1}\wtd{\hB}(b)^{-1})^{-1},\wtd{\hB}\left(\wtd{\hB}(b)\hB(a)\right)\right)\\
=&\left(\hB\left(\hB(a)^{-1}\wtd{\hB}(b)^{-1}\right)^{-1},\wtd{\hB}\left((\hB(a)^{-1}\wtd{\hB}(b)^{-1})^{-1}\right)\right).
\end{aligned}
\end{equation}
Then it follows from the proof Proposition \ref{C3}, \eqref{C1} and \eqref{C2} that
$$\hC\left(\hB(a),\wtd{\hB}(b)\right) \wtd{\hC}\left((\hB(a),\wtd{\hB}(b))\right)=\wtd{\hC}\left((\hB(a),\wtd{\hB}(b))\right)\hC\left(\hB(a),\wtd{\hB}(b)\right) =(\hB(a),\wtd{\hB}(b)).$$ Finally, by Proposition \ref{GMEQ} and Lemma \ref{TE2}, we obtain that $\wtd{\hC}$ is an idempotent group homomorphism and $\hC$ is a group projection on $(G_+,G_-,\rho,\mu)$.
\end{proof}

Here is the group version of Corollary \ref{RBIS}.
\begin{cor}
With the notations given in Theorem \ref{TMG}, denote $G_1 =\im \hC$ and $G_2=\im \wtd{\hC}.$ Let $p_1: G_1\oplus G_2\to G_1$ be the projection. Then $(G_1\oplus G_2,p_1)$ is Rota-Baxter Lie algebra of weight $-1$, which is Rota-Baxter isomorphic to $(G_+\bowtie G_-,C).$
\end{cor}

\begin{proof}
The proof is similar to the proof of Corollary \ref{RBIS}.
\end{proof}

Let $(G,\hB)$ and $(G',\hB')$ be Rota-Baxter groups of weight $-1.$ Recall that a Rota-Baxter group homomorphism is a map $f:(G,\hB)\to (G',\hB')$ such that 
$f$ is a group homomorphism from $G$ to $G'$ and $f\circ \hB=\hB' \circ f .$

By the definition of $\hC$, it is not hard to check that 
$$ \im \hC=\Big\{(\hB(a^{-1})^{-1},\wtd{\hB}(a))|a\in G\Big\}.$$ 

Here, we give the group version of Corollary \ref{BBIS}.
\begin{cor}
With the notations given in Theorem \ref{TMG}, let $G_1=\im \hC.$ Define $\hB_1:G_1\to G_1$ by 
$$\hB_1\left((\hB(a^{-1})^{-1},\wtd{\hB}(a))\right)=\left(\hB(\wtd{\hB}(a)^{-1}),\wtd{\hB}(\wtd{\hB}(a))  \right), \ \forall a\in G$$
and $\wtd{\hB_1}:G_1\to G_1$ by 
$$ \wtd{\hB_1}\left((\hB(a^{-1})^{-1},\wtd{\hB}(a))\right)=\left(\hB(\hB(a^{-1})),\wtd{\hB}(\hB(a^{-1})^{-1}) \right),\ \forall a\in G.$$
Then:
\begin{itemize}
\item[(a)]$(G_1,\hB_1)$ is a Rota-Baxter group of weight $-1$ that is Rota-Baxter isomorphic to $(G,\wtd{\hB});$
\item[(b)]
$(G_1,\wtd{\hB_1})$ is a Rota-Baxter group of weight $-1$ that is Rota-Baxter isomorphic to $(G,\hB).$
\end{itemize}
\end{cor}
\begin{proof}

First, it is not hard to see that $\hB_1$ is well defined.
Then it follows from Proposition \ref{G-1} and Lemma \ref{SLS} that $$\hB_1\left((\hB(a^{-1})^{-1},\wtd{\hB}(a))\right)=\hC\left( (e,b)\right)$$ and
$$\wtd{\hB_1}\left((\hB(a^{-1})^{-1},\wtd{\hB}(a))\right)=(a,b)\hC((e,b)^{-1})$$ for any $a\in G.$
It follows from Proposition \ref{GS} that $\hB_1$ and $\wtd{\hB_1}$ are Rota-Baxter operators of weight $-1$ on $G$.
Define $\pi:G\to G_2$ by
$$\pi(a)=(\hB(a^{-1})^{-1},\wtd{\hB}(a)).$$
For any $a\in G,$ if $\pi(a)=(e,e),$ it follows that $\hB(a^{-1})^{-1}=\wtd{\hB}(a)=e.$ By \eqref{HBWB}, we have $$a=a\wtd{\hB}(a)^{-1}\wtd{\hB}(a)=\hB(a^{-1})\wtd{\hB}(a)=e.$$
This proves $\pi$ is injective. It follows from the definition of $\pi$ that $\pi$ is surjective. Then we show that $\pi$ is a group homomorphism. By Lemma \ref{SLS}, we have
\begin{equation*}
\begin{aligned}
\pi(a)\pi(b)=&\left(\hB(a^{-1})^{-1},\wtd{\hB}(a)\right)\left(\hB(b^{-1})^{-1},\wtd{\hB}(b)\right)\\
=&\left(\hB(b^{-1}a^{-1})^{-1}, \wtd{\hB}(ab)\right)=\pi(a,b)
\end{aligned}
\end{equation*}
for any $a,b\in G$.
Then we show that $\pi$ is compatible with $\wtd{\hB}$ and $\hB_1$. For any $a\in G$, we have
$$\pi(\wtd{\hB}(a))= \left(\hB(\wtd{\hB}(a)^{-1})^{-1},\wtd{\hB}(\wtd{\hB}(a))\right)=\hB_1((\hB(a^{-1})^{-1},\wtd{\hB}(a)))=\hB(\pi(a)).$$
This means that $\pi$ is a Rota-Baxter isomorphism from $(G,\wtd{\hB})$ to $(G_1,\hB_1).$
It is similar to show (b).

\end{proof}

Let $(G,\hB)$ be a Rota-Baxter group of weight $-1.$ By \cite{LG}, we know that $H_+$ and $H_-$ are both ideals of $G.$
Denote $H_+H_-$ the subgroup of $G_{\hB}$ generated by $H_+$ and $H_-,$ that is,
$$H_+H_-=\Big\{ a_1\cdot_{\hB} b_1 \cdot_{\hB} a_2\cdot_{\hB} b_2\cdot_{\hB} a_n \cdot_{\hB}b_n|a_i\in H_+,\ b_i \in H_-  \Big\}.$$ As $H_+$ and $H_-$ are both ideals of $G_{\hB}$ and $G_{\wtd{B}}$, hence $H_+H_-$ is also an ideal of $G_{\hB}$ and an ideal of $G_{\wtd{\hB}}.$

The next lemma will be used in the proof of Proposition \ref{OL}.

\begin{lem}\label{DRB}
Let $(G,\hB)$ be a Rota-Baxter group of weight $-1.$ Then $\hB$ is a Rota-Baxter operator of weight $-1$ on $G_{\hB}.$ 
\end{lem}
\begin{proof}
For any $a,b\in G,$ by \eqref{RBLS}, we have 
\begin{equation*}
\hB(a)\cdot_{\hB} \hB(b)=(\Ad_{\hB\circ \hB(a)}(\hB(b)))\hB(a)=\hB\left( \hB(a)\cdot_{\hB} b \cdot_{\hB}\hB(a)^{-1} \cdot_{\hB} a    \right).
\end{equation*}
This shows that $\hB$ is indeed a Rota-Baxter operator on $G_{\hB}.$
\end{proof}
The Rota-Baxter group of weight $-1$ $(G_{\hB},\hB)$ given in the above lemma is called the descendent Rota-Baxter group of $(G,\hB).$

Then we show that there is a Rota-Baxter group structure on $G_{\hB}/(H_+ H_-).$

\begin{pro}\label{OL}
Let $(G,\hB)$ be a Rota-Baxter group of weight $-1.$ Define the operator $\ol{\hB}:G_{\hB}/(H_+ H_-)\to G_{\hB}/(H_+ H_-)$ by
$$\ol{\hB}(\ol{a})=\ol{\wtd{\hB}(a)}, \ \forall a\in G$$and the operator 
$\wtd{\ol{\hB}}:G_{\wtd{\hB}}/(H_+ H_-)\to G_{\wtd{\hB}}/(H_+ H_-)$ by
$$\wtd{\ol{\hB}}(\ol{a})=\ol{\wtd{\hB}(a)}, \ \forall a\in G.$$
Then 
$(G_{\hB}/(H_+ H_-), \ol{\hB})$ and $(G_{\wtd{\hB}}/(H_+ H_-),\wtd{\ol{B}})$ are Rota-Baxter groups of weight $-1.$
\end{pro}
\begin{proof}
It follows directly from the definition that $\ol{\hB}$ is well defined. Then by Lemma \ref{DRB}, we know $(G_{\hB}/(H_+ H_-), \ol{\hB})$ is a Rota-Baxter group of weight $-1.$
\end{proof}

Finally, in the next theorem, we prove 
that there is a Rota-Baxter structure on $\im \wtd{C},$ and it is Rota-Baxter isomorphic to $(G_{\hB}/(H_+ H_-), \ol{\hB})$. Where $\wtd{C}$ is given in Theorem \ref{TMG}.
\begin{thm}\label{FTH}
Let $(G,\hB)$ be a Rota-Baxter group of weight $-1$ and $(G_+,G_-,\rho,\mu)$ be the matched pair of groups on $(G,\hB).$ Let $\wtd{\hC}$ be the operator defined in Theorem \ref{TMG}, $G_2=\im \wtd{\hC}$ and $\ol{\hB}:G_{\hB}/(H_+ H_-)\to G_{\hB}/(H_+ H_-)$ be the operator defined in Proposition \ref{OL}. Define the operator $\hB_2: G_2\to G_2$ by 
$$\hB_2((a,b))=\wtd{\hC}((e,b)), \ \forall (a,b)\in G_2,$$ and the operator $\wtd{\hB_2}:G_2\to G_2$ by
$$\wtd{\hB_2}((a,b))=\wtd{\hC}((a,e)), \ \forall (a,b)\in G_2.$$
Then the following holds:
\begin{itemize}
\item[(a)]$(G_2,\hB_2)$ is a Rota-Baxter group of weight $-1$ which is Rota-Baxter isomorphic to $\left(G_{\hB}/(H_+ H_-), \ol{\hB}\right);$ 
\item[(b)]
$(G_2,\wtd{\hB_2})$ is a Rota-Baxter group of weight $-1$ which is Rota-Baxter isomorphic to $\left(G_{\wtd{\hB}}/(H_+ H_-), \wtd{\ol{\hB}}\right).$
\end{itemize}
\end{thm}

\begin{proof}
By Proposition \ref{GS}, we know that $(G_2,\hB_2)$ is a Rota-Baxter group of weight $-1.$ Now define the operator $\pi:(G_{\hB}/(H_+ H_-), \ol{\hB})\to (G_2, \hB_2)$ by 
$$ \pi(\ol{a})=\left(\hB(\wtd{\hB}(a^{-1})  )^{-1},\wtd{\hB}(\hB(a)^{-1})\right),\ \forall a\in G.$$
It follows directly from the definition that $\pi$ is well defined. Then we prove that $\pi$ is a group homomorphism. For any $a,b\in G,$ we have 
$$    
\begin{aligned}
\pi(\ol{a})\pi(\ol{b})&=\left(\hB(\wtd{\hB}(a^{-1})  )^{-1},\wtd{\hB}(\hB(a)^{-1})\right) \left(\hB(\wtd{\hB}(b^{-1})  )^{-1},\wtd{\hB}(\hB(b)^{-1})\right)\\
&=\left((\hB(\Ad_{\wtd{\hB}(\hB(b)^{-1})^{-1}}  \wtd{\hB}(a^{-1})  ))^{-1}\hB(\wtd{\hB}(b^{-1})  )^{-1}, \wtd{\hB}(\hB(a)^{-1})\wtd{\hB}(\Ad_{\hB(\wtd{\hB}(a^{-1})  )^{-1}} \hB(b)^{-1})  \right)\\
&=\left((\hB(\wtd{\hB}(a^{-1})\wtd{\hB}(b^{-1}) ))^{-1}, \wtd{\hB}(\hB(b)^{-1}\hB(a)^{-1})\right)(\textit{by Lemma \ref{IS} and $\eqref{RBLS}$})\\
&=\left(\hB(\wtd{\hB}((a\cdot_{\hB} b)^{-1}))^{-1},\wtd{\hB}(\hB(a\cdot_{\hB} b)^{-1})\right)(\textit{by Lemma \ref{DEC} and \eqref{RBLS}})\\
&=\pi(\ol{a\cdot_{\hB} b}).
\end{aligned}
$$
This implies $\pi$ is a Lie algebra homomorphism. Then we prove that $\pi$ is injective. For any $a\in G,$ if $\pi(\ol{a})=(e,e),$ then we have $\hB(\wtd{\hB}(a^{-1}))^{-1}=e.$ It follows that $\wtd{\hB}(a^{-1})\in H_-.$ Similarly, as $\wtd{\hB}(\hB(a)^{-1})=e,$ we have $\hB(a)\in H_+.$ It follows that $a^{-1}=\wtd{\hB}(a)\hB(a^{-1})^{-1}\in H_+ H_-.$ This means $a\in H_+ H_-.$ Then from the definition of $\pi$, it follows that $\pi$ is surjective. Finally, we show that $\pi$ is compatible with the Rota-Baxter operators. For any $a\in G,$ we have 
$$\pi(\ol{\hB(a)})=\left(\hB(\wtd{\hB}(\hB(a)^{-1})  )^{-1},\wtd{\hB}(\hB(\hB(a))^{-1})\right)=\wtd{C}\left(e,\wtd{\hB}(\hB(a)^{-1})\right)=\hB_2(\pi(\ol{a})).$$ Therefore, $\pi$ is a Rota-Baxter group isomorphism. This proves (a), and it it similar to prove (b).
\end{proof}

It follows directly from the above theorem to get the following corollary.
\begin{cor}
With the notations given in Theorem \ref{FTH}, define $\pi_1:(G_{\hB}, \hB)\to (G_2,\hB_2)$ and $\pi_2:(G_{\wtd{\hB}},\wtd{\hB})\to (G_2,\wtd{\hB_2})$ by
\begin{equation*}
\begin{aligned}
\pi_1(a)&=\left(\hB(\wtd{\hB}(a^{-1})  )^{-1},\wtd{\hB}(\hB(a)^{-1})\right),\\
\pi_2(a)&=\left(\hB(\wtd{\hB}(a^{-1})  )^{-1},\wtd{\hB}(\hB(a)^{-1})\right),
\end{aligned}
\end{equation*}
for any $a\in  G.$ Then $\pi_1$ and $\pi_2$ are surjective Rota-Baxter group homomorphisms.
\end{cor}

\vspace{2mm}
\noindent
{\bf Acknowledgements. } This research is supported by Scientific Research Foundation for High-level Talents of Anhui University of Science and Technology (2024yjrc49). We thank Yufei Qin for his helpful suggestions.

\smallskip

\vspace{2mm}
\noindent
{\bf Declaration of interests. } The authors have no conflicts of interest to disclose.

\smallskip

\vspace{2mm}
\noindent
{\bf Data availability. } Data sharing not applicable to this article as no datasets were generated or analysed in this work.

\vspace{-.2cm}



\begin{thebibliography}{ab}

\bibitem{BX} G. Baxter, An analytic problem whose solution follows from a simple algebraic identity, \emph{Pacific J. Math.} {\bf 10} (1960), 731-742.

\bibitem{BD} A. A. Belavin and V. G. Drinfeld, Solutions of the classical Yang-Baxter equation for simple Lie algebras, {\it Funct. Anal. Appl.} {\bf 16} (1982), 159-180.

\bibitem{FC} F. Catino, M. Mazzotta and P. Stefanelli, Rota-Baxter operators on Clifford semigroups and the Yang-Baxter equation, \emph{J. Algebra} {\bf 622}(2023), 587-613.

\bibitem{CO} A. Connes and D. Kreimer, Renormalization in quantum field theory and the Riemann-Hilbert problem I. The Hopf algebra structure of graphs and the main theorem, \emph{Comm. Math. Phys.} {\bf 210} (2000), 249-273. 

\bibitem{DA} A. Das and N. Rathee, Extensions and automorphisms of Rota-Baxter groups, \emph{J. Algebra} {\bf 636} (2023), 626-665.

\bibitem{D} V. G. Drinfeld, Hamiltonian structures on Lie groups, Lie bialgebras and the geometric meaning of the classical Yang-Baxter equations, {\it Soviet Math. Dokl.} {\bf 27} (1983), 68-71.

\bibitem{FR} E. Frenkel, N. Reshetikhin and M. A. Semenov-Tian-Shansky, Drinfeld-Sokolov reduction for difference operators and deformations of $W$-algebras I. The case of Virasoro algebra, \emph{Comm. Math. Phys.} {\bf 192} (1998), 605-629.

\bibitem{LG} L. Guo, H. Lang and Y. Sheng, Integration and geometrization of Rota-Baxter Lie algebras, {\it Adv. Math.} {\bf 387} (2021), no. 107834.

\bibitem{KO} Y. Kosmann-Schwarzbach and F. Magri, Poisson Lie groups and complete integrability, I, \emph{Ann. Inst. H. Poincar\'e} {\bf 49} (1988), 433-460.

\bibitem{K} Y. Kosmann-Schwarzbach, Lie bialgebras, Poisson Lie groups and dressing transformation, in ``Integrability of nonlinear systems'', Lecture Notes in Physics {\bf 495}, Springer, Berlin (1997), 104-170.


\bibitem{Ku}
B. A. Kupershmidt, What a classical $r$-matrix really is, \emph{J. Nonlinear Math. Phys.} {\bf 6} (1999), 448-488.

\bibitem{Hl} H. Lang and Y. Sheng, Factorizable Lie bialgebras, quadratic Rota-Baxter Lie algebras and Rota-Baxter Lie
bialgebras,  \emph{Comm. Math. Phys.} {\bf 397} (2022), 763-791.


\bibitem{LU}J. H. Lu and A. Weinstein, Poisson Lie groups, dressing transformations and Bruhat
decomposition, \emph{J. Differential Geom.} {\bf 31} (1990), 501-526.

\bibitem{MJ} S. Majid, Matched pairs of Lie groups associated to solutions of the Yang-Baxter equation,  \emph{Pacific J. Math.} {\bf 141}
(1990), 311-332.


\bibitem{GR} A. G. Reyman and M. A. Semenov-Tian-Shansky, Reduction of Hamilton systems, affine Lie algebras and Lax
equations, \emph{Invent. Math.} {\bf 54} (1979), 81-100.

\bibitem{GR1} A. G. Reyman and M. A. Semenov-Tian-Shansky, Reduction of Hamilton systems, affine Lie algebras and Lax
equations II, \emph{Invent. Math.} {\bf 63} (1981), 423-432.


\bibitem{RT} G.-C. Rota, Baxter algebras and combinatorial identities, I, II, \emph{Bull. Amer. Math. Soc.} {\bf 75} (1969), 325-329.








\bibitem{STS} M. A. Semenov-Tyan-Shanskii, What is a classical $r$-matrix? \emph{Funct. Anal. Appl.} {\bf 17} (1983), 259-272.

\bibitem{S2} M. A. Semenov-Tyan-Shanskii, Integrable systems and factorization problems, {\it Operator Theory: Advances and Applications} {\bf141} (2003), 155-218.

\bibitem{VG} V. G. Bardakov, V. Gubarev,
Rota-Baxter groups, skew left braces, and the Yang-Baxter equation,
\emph{J. Algebra} {\bf 596}
(2022), 328-351.
\end{thebibliography}
\end{document}